\documentclass[11pt]{amsart}

\usepackage{amsmath,hyperref}
\usepackage{amsfonts}
\usepackage{mathrsfs}
\usepackage{amssymb}
\usepackage{amsthm}
\usepackage{mathtools}
\usepackage{graphicx}
\usepackage{setspace}
\usepackage{enumerate}
\usepackage{tikz}
\usepackage{url,color}
\usepackage[paperheight=11in,paperwidth=8.5in,margin=1in]{geometry}
\usepackage{chngcntr}
\usepackage{multicol}

\theoremstyle{plain}
\newtheorem{thm}{Theorem}[section] 
\newtheorem{cor}[thm]{Corollary}
\newtheorem{prop}[thm]{Proposition}
\newtheorem{lem}[thm]{Lemma}
\newtheorem{conj}[thm]{Conjecture}
\newtheorem{op}[thm]{Open Problem}
\counterwithin{figure}{section}
\counterwithin{table}{section}

\theoremstyle{definition}
\newtheorem{exmp}[thm]{Example}
\newtheorem{rmk}[thm]{Remark}

\newcommand{\fS}{{\mathfrak S}}

\DeclareMathOperator{\shift}{shift}
\DeclareMathOperator{\Bshift}{Bshift}
\DeclareMathOperator{\Cshift}{Cshift}

\DeclareMathOperator{\Shuff}{Shuff}

\DeclareMathOperator{\Ballot}{Ballot}
\DeclareMathOperator{\ballot}{ballot}

\DeclareMathOperator{\RBallot}{RBallot}

\definecolor{vividviolet}{rgb}{0.62, 0.0, 1.0}
\definecolor{OliveGreen}{rgb}{0.0, .6, .1}

\newcommand{\tg}{\textcolor{gray}}

\newcommand{\diam}{\text{diam}}

\newcommand{\ol}{\overline}
\newcommand{\ul}{\underline}

\hypersetup{
    colorlinks=true,
    allcolors={blue}
}

\author{Samantha Dahlberg}
\address{
Department of Mathematics,
Arizona State University,
Tempe, AZ, USA}
\email{sdahlber@asu.edu}

\author{Younghwan Kim}
\address{
Department of Mathematics,
Arizona State University,
Tempe, AZ, USA}
\email{ykim117@asu.edu}

\title{Diameters of Graphs on Reduced Words of 12 and 21-Inflations}
\begin{document}
\maketitle


\begin{abstract}
It is a classical result that any permutation in the symmetric group can be generated by a sequence of adjacent transpositions. The sequences of minimal length are called reduced words, and in this paper we study the graphs of these reduced words, with edges determined by relations in the underlying Coxeter group. Recently, the diameter has been calculated for the longest permutation $n\ldots 21$ by Reiner and Roichman as well as Assaf. In this paper we find inductive formulas for the diameter of the graphs of 12-inflations and many 21-inflations. These results extend to the associated graphs on commutation and long braid classes. Also, these results give a recursive formula for the diameter of the longest permutation, which matches that of Reiner, Roichman and  Assaf. Lastly, We  make progress on conjectured bounds of the diameter by Reiner and Roichman, which are based on the underlying hyperplane arrangement, and find families of permutations that achieve the upper bound and lower bound of the conjecture. In particular permutations that avoid 312 or 231 have graphs that achieve the upper bound. 
\end{abstract}
\section{Introduction}
\label{sec:intro}

The symmetric group $\fS_n$  of $[n]:=\{1,2,\ldots, n\}$ can be generated by adjacent transpositions $s_i=(i,i+1)$. The shortest sequences of adjacent transpositions that achieve $\pi\in\fS_n$ are called reduced words of $\pi$, the collection of which is denoted by $R(\pi)$. Tits~\cite{T69} showed that one can transform any reduced word of $\pi$ into any other by a sequence of 
\begin{enumerate}
\item commutation moves, where you exchange adjacent $i$ and $j$ if $|i-j|>1$, and 
\item long braid moves, where you exchange adjacent sequences $j(j+1)j$ and $(j+1)j(j+1)$. 
\end{enumerate}
The graphs $G_{\pi}$ formed by vertices $R(\pi)$ and edges associated to a single commutation or a single long braid move, called a commutation or long braid edge respectively, have been well studied~\cite{B99, EG87, E95,GL02,Stem96,Stem98}.  
Tits showed that the graph is connected~\cite{T69}. Stanley~\cite{Stanley84} enumerated $R(\pi)$ using a symmetric functions. These symmetric functions have led to connections to Schubert calculus, Demazure characters and flag skew Schur functions~\cite{AS18,A19_2,BJS93,MS15}. 
The associated graph $C_{\pi}$, which is $G_{\pi}$ contracted along commutation edges, has additionally received a lot of attention. Elnitsky proved that the vertices of $C_{\pi}$ are in bijection with rhombic tilings of certain polygons and that the graph is bipartite~\cite{E94}. The graph $C_{\pi}$ also has connections to geometric representation theory~\cite{CKL20}.

Interestingly, the theory of permutation pattern enumeration and avoidance has found itself highly useful in describing properties of $R(\pi)$ and its associated graphs $G_{\pi}$ and $C_{\pi}$~\cite{Daly, Meng, Tenner06, Tenner12, Tenner17}.  Vexillary permutations, those avoiding 2143, have $R(\pi)$ enumerated by the number of standard Young tableaux of a single shape~\cite{Stanley84}, and permutations avoiding 321 have only have commutation edgs~\cite{BJS93}. 

The longest permutation $\delta_n=n(n-1)\cdots 1$ has been particularly well studied. This is because $R(\pi)$ is closely connected to the weak Bruhat order and type A hyperplane arrangements. Reiner found that in $G_{\delta_n}$, the expected number of long braid edges is one~\cite{R05}. The expected number of long braid edges is also one even inside a single commutation class, those vertices connected by commutation edges~\cite{STWW17}. Tenner studied the expected number of commutation edges, which is more complex~\cite{Tenner15}.

The diameter of $G_{\delta_n}$ was first been determined by Reiner and Roichman~\cite{RR13}, using hyperplane arrangements, and  later by Assaf~\cite{A19}, using posets and Rothe diagrams. The diameter of $C_{\delta_n}$ has also been calculated~\cite{GMS20}. Our paper continues this direction and finds the diameters of $G_{\pi}$, $C_{\pi}$ and $B_{\pi}$ in several cases, where $B_{\pi}$ is $G_{\pi}$ contracted along all long braid edges. Our project is motivated by the conjectured upper and lower bounds for the diameter of $G_{\pi}$ by Reiner and Roichman~\cite{RR13}. While we haven't verified the conjectured bounds in all cases, we are able to prove these bounds in all cases we have a formula for the  diameter of $G_{\pi}$. 

Our paper is organized as follows. We first introduce preliminary topics in Section~\ref{sec:pre} where we also define the two families of permutations we are studying, 12-inflations and 21-inflations. In Sections~\ref{sec:12-inflations} and \ref{sec:21-inflations} we describe a way to encode the vertices $R(\pi)$, so we can more easily construct paths in $G_{\pi}$. With these paths we find exact recursive formulas for the diameters of $G_{\pi}$, $C_{\pi}$ and $B_{\pi}$ in the case of 12-inflations in Theorem~\ref{thm:12_formula}, and we find recursive upper and lower bounds for the diameters of graphs $G_{\pi}$, $C_{\pi}$ and $B_{\pi}$ in the case of certain 21-inflations in Theorem~\ref{thm:21-bounds}. Also, we prove several graph isomorphism using symmetries of the square in Section~\ref{sec:symm}.  We are then able to leverage our inductive formulas to find the exact diameters of $G_{\pi}$, $C_{\pi}$ and $B_{\pi}$ in the case of permutations $\pi$ that avoid 231 or 312 in Section~\ref{sec:pattern_avoidance}. Finally, in Section~\ref{sec:RRconjecture}, we will connect our results back to Reiner and Roichman's conjectured upper and lower bounds. In particular we prove that all permutations $\pi$ that avoid 231 or 312 satisfy this conjecture, and that these permutations achieve the upper bound of the conjecture in Theorem~\ref{thm:312-upper} and Theorem~\ref{thm:231-upper}. Additionally, we describe another infinite family of permutations that achieves the lower bound of the conjecture in Theorem~\ref{thm:low_bound}. We end our paper with Section~\ref{sec:future} where we consider further directions. 

\section{Preliminaries}
\label{sec:pre}
In this section we will give background information. The definitions are mainly obtained from Bj\"{o}rner-Brenti~\cite{BB}, Bollob\'{a}s~\cite{Bol}, and Stanley~\cite{Sta12,Sta99}.

Let $\fS_n, n\geq 0$ denote the symmetric group on $[n]:=\{1,2,\dots,n\}$. A permutation $\pi\in\fS_n$ permutes $[n]$ by mapping $i\mapsto \pi_i$. We write $\pi$ in one-line notation $\pi=\pi_1\pi_2\dots \pi_n$. For $\pi\in\fS_n$ we say that $n$ is the \textit{size} of $\pi$ and denote it by $|\pi|$. 

The symmetric group $\fS_n$ is generated by the set of simple reflections $S=\{s_1,s_2,\dots,s_{n-1}\}$, where $s_i$ interchanges $i$ and $i+1$. Thus, for $\pi\in \fS_n$, a decomposition $\pi=s_{i_1}s_{i_2}\dots s_{i_k}$ with letters in $S$ is called a \textit{reduced decomposition} for $\pi$ if $k$ is minimal. The word $i_1i_2\dots i_k$ is called a \textit{reduced word} for $\pi$. We say that $k$ is the \textit{length} of $\pi$ and denote it by $\ell(\pi)$. For a permutation $\pi\in \fS_n$ we will use $R(\pi)$ for the set of reduced words for $\pi$. Two reduced words for $\pi\in\fS_n$ can be related by
\begin{enumerate}
\item {\it short braid moves} or {\it commutation moves} by switching adjacent $jk$ if $|j-k|>1$ or
\item {\it long braid moves} by exchanging the occurrences of $j(j+1)j$ and $(j+1)j(j+1)$ on consecutive indices.  
\end{enumerate}
We say two reduced words for $\pi\in\fS_n$ are in the same \textit{commutation class} if we can obtain one from another by applying a sequence of commutation moves. \textit{Braid classes} of $\pi$ are defined in terms of the longbraid moves.

An \textit{inversion} in $\pi$ is a pair $(i,j)$ such that $i<j$ and $\pi_i>\pi_j$. Since the number of inversions in $\pi$ is equal to its length $\ell(\pi)$, the longest permutation in $\fS_n$ is $\delta_n=n\dots 21\in\fS_n$ and the shortest permutation is the identity $\iota_n=12\dots n\in\fS_n$.

\begin{exmp}
For $\pi=4231\in\fS_4$ we have $R(\pi)=\{12321,31231,13231,31213,13213,32123\}$, $|\pi|=4$ and  $\ell(\pi)=5$. We see that there are $\ell(\pi)=5$ inversions in $\pi$, which are $(1,2), (1,3), (1,4), (2,4)$ and  $(3,4)$.
\end{exmp}

Consider two sequences of integers $u=u_1u_2\ldots u_n$ and $v=v_1v_2\ldots v_n$. We say that $u$ and $v$ are {\it order isomorphic} if $u_i\leq u_j$ ($u_i\geq u_j$) if and only if $v_i\leq v_j$ ($v_i\geq v_j$). For example $463$ and $231$ are order isomorphic. We say a permutation $\pi$ {\it contains} a pattern $\sigma$ if $\pi$ has a subsequence that is order isomorphic to $\sigma$. We say a permutation $\pi$ {\it avoids} $\sigma$ if $\pi$ does not contain $\sigma$. 

\begin{exmp}
The permutation $\pi =142635$ contains a pattern $231$ because $\pi$ has the subsequence $463$ and avoids $321$ since there is no decreasing subsequence of length three. 
\end{exmp}

Consider a permutation $\pi\in\fS_n$. A {\it block} of $\pi$ is a consecutive sequence of $\pi_a\pi_{a+1}\ldots\pi_{b}$ whose union of values forms a consecutive interval of integers, $\{\pi_a,\pi_{a+1},\ldots,\pi_{b}\}=[c,d]=\{c,c+1,\ldots, d\}$ for some integers $c\leq d$. Let $\sigma\in\fS_k$ and $\pi_{(1)},\pi_{(2)},\ldots,\pi_{(k)}$ be permutations of possibly different non-negative lengths. The inflation of $\pi_{(1)},\pi_{(2)},\ldots,\pi_{(k)}$ by $\sigma$, written $\sigma[\pi_{(1)},\pi_{(2)},\ldots,\pi_{(k)}]$, is $\sigma$, but we replace $\sigma_i$ with a block order isomorphic to $\pi_{(i)}$ for $1\leq i\leq k$. See Figure~\ref{fig:inflation_ex} for an example.


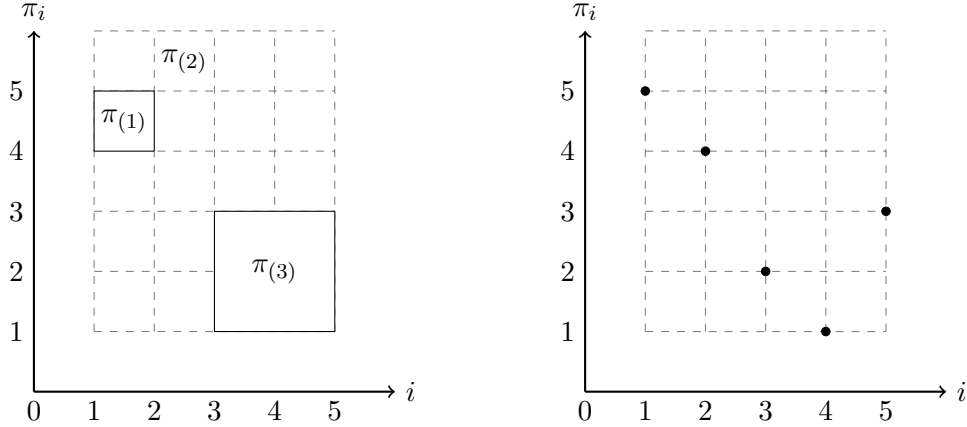
\begin{figure}
\begin{center}
\begin{tikzpicture}[scale=0.8]
\node[left, below] at (0,0) {$0$};
\node[below] at (1,0) {$1$};
\node[below] at (2,0) {$2$};
\node[below] at (3,0) {$3$};
\node[below] at (4,0) {$4$};
\node[below] at (5,0) {$5$};

\node[left] at (0,1) {$1$};
\node[left] at (0,2) {$2$};
\node[left] at (0,3) {$3$};
\node[left] at (0,4) {$4$};
\node[left] at (0,5) {$5$};

\draw[->,thick] (0,0)--(6,0) node[right]{$i$};
\draw[->,thick] (0,0)--(0,6) node[above]{$\pi_i$};
\draw[gray, dashed] (1,1) grid (5,6);
\filldraw[fill=white, draw=black] (1,4) rectangle (2,5);
\node at (1.5,4.5) {$\pi_{(1)}$};
\node at (2.5,5.5) {$\pi_{(2)}$};
\filldraw[fill=white, draw=black] (3,1) rectangle (5,3);
\node at (4,2) {$\pi_{(3)}$};

\end{tikzpicture}
\qquad
\qquad
\begin{tikzpicture}[scale=0.8]
\node[left, below] at (0,0) {$0$};
\node[below] at (1,0) {$1$};
\node[below] at (2,0) {$2$};
\node[below] at (3,0) {$3$};
\node[below] at (4,0) {$4$};
\node[below] at (5,0) {$5$};

\node[left] at (0,1) {$1$};
\node[left] at (0,2) {$2$};
\node[left] at (0,3) {$3$};
\node[left] at (0,4) {$4$};
\node[left] at (0,5) {$5$};

\draw[->,thick] (0,0)--(6,0) node[right]{$i$};
\draw[->,thick] (0,0)--(0,6) node[above]{$\pi_i$};
\draw[gray, dashed] (1,1) grid (5,6);
\filldraw (1,5) circle (2pt);
\filldraw (2,4) circle (2pt);
\filldraw (3,2) circle (2pt);
\filldraw (4,1) circle (2pt);
\filldraw (5,3) circle (2pt);
\end{tikzpicture}
\end{center}
\caption{The inflation of $\pi_{(1)}=21,\pi_{(2)}=\epsilon, \pi_{(3)}=213$ by $\sigma=231$ is $\pi=\sigma[\pi_{(1)},\pi_{(2)},\pi_{(3)}]=54213$ where $\epsilon\in \fS_0$. The diagram is illustrated above. }
\label{fig:inflation_ex}
\end{figure}

Let $G=(V,E)$ be a simple, connected graph with vertex set $V$ and edge set $E$. The \textit{distance} $d_G(u,v)$ between $u,v\in V$ is the number of edges in the shortest path between $u$ and $v$. The \textit{diameter} $\diam(G)$ of the graph $G$ is the maximum distance between any two vertices in $G$. A subgraph $G'=(V',E')$ is called an \textit{induced subgraph} if $G'$ contains all edges of $G$ that join two vertices in $V'$. 

For $\pi\in \fS_n$ we define $G_{\pi}$ to be the graph on $R(\pi)$ where edges come from commutation and long braid moves, which we will refer to as {\it commutation edges} and {\it long braid edges} respectively. 
Denote $C_{\pi}$ to be the graph $G_{\pi}$ where we contract along commutation edges, so the only edges remaining are long braid edges and the vertices are commutation classes. Let $B_{\pi}$ be the graph where we contract $G_{\pi}$ along long braid edges, so the only remaining edges remaining are commutation edges and the vertices are braid classes.  

\begin{exmp}
Let $\pi=4231\in\fS_4$. The graphs $G_{\pi},C_{\pi}$, and $B_{\pi}$ are in Figure~\ref{fig:graphs}. We see that $\diam(G_{\pi})=4, \diam(C_{\pi})=2$, and $\diam(B_{\pi})=2$.
\end{exmp}

\begin{figure}

\begin{center}
\begin{tikzpicture}
	\node[align=center] (G) at (0,6) {$G_{\pi}:$};
	\node[align=center] (G12321) at (2,6) {$12321$};
	\node[align=center] (G13231) at (4,6) {$13231$};
	\node[align=center] (G31231) at (6,5) {$31231$};
	\node[align=center] (G13213) at (6,7) {$13213$};
	\node[align=center] (G31213) at (8,6) {$31213$};
	\node[align=center] (G32123) at (10,6) {$32123$};
	\draw (G13231)--(G31231)--(G31213)--(G13213)--(G13231);
	\draw [double] (G12321)--(G13231);
	\draw [double] (G31213)--(G32123);
		
	\node[align=center] (C) at (0,3) {$C_{\pi}:$};
	\node[align=center] (C12321) at (2,3) {$\{12321\}$};
	\node[align=center] (C13231) at (6,3) {$\{13231,31231,13213,31213\}$};
	\node[align=center] (C32123) at (10,3) {$\{32123\}$};
	\draw [double] (C12321)--(C13231)--(C32123);
	
	\node[align=center] (B) at (0,0) {$B_{\pi}:$};
	\node[align=center] (B12321) at (2,0) {$\{12321,13231\}$};
	\node[align=center] (B31231) at (6,-1) {$\{31231\}$};
	\node[align=center] (B13213) at (6,1) {$\{13213\}$};
	\node[align=center] (B31213) at (10,0) {$\{31213,32123\}$};
	\draw (B12321)--(B31231)--(B31213)--(B13213)--(B12321);
\end{tikzpicture}
\end{center}
\caption{The graphs $G_{\pi},C_{\pi},$ and $B_{\pi}$ where $\pi=4231\in\fS_4$}
\label{fig:graphs}
\end{figure}
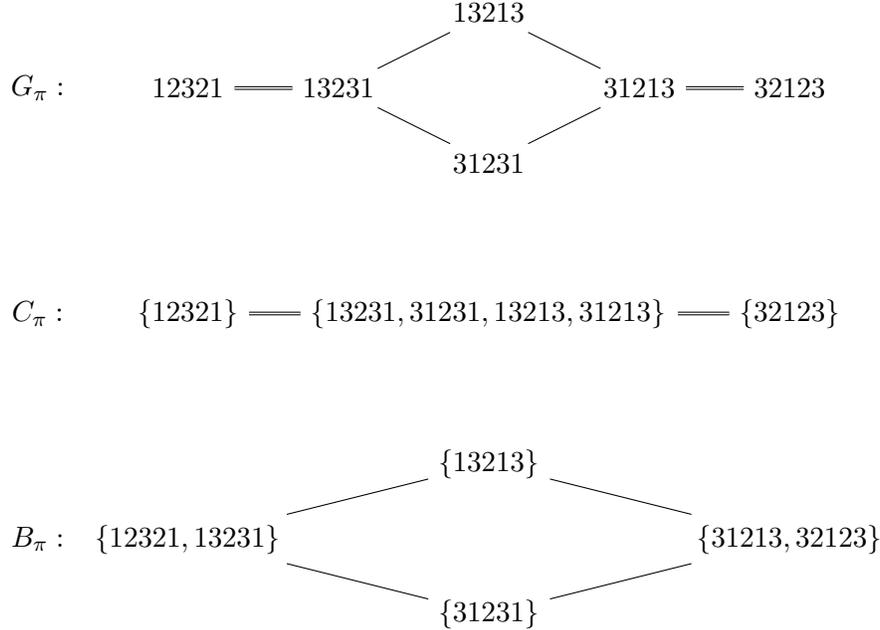

\begin{lem}
\label{lem:diam_bounds}
Let $\pi\in\fS_n$ and $w,w'\in R(\pi)$. 
\begin{enumerate}
\item If for any $r,r'\in R(\pi)$ we can find a path from $r$ to $r'$ with at most $c$ commutation moves plus at most $b$ long braid moves,
then $\diam(G_{\pi})\leq b+c$, $\diam(C_{\pi})\leq b$ and $\diam(B_{\pi})\leq c$. 
\item If we can find a pair $r,r'\in R(\pi)$ where any path from $r$ to $r'$ has at least $c$ commutation moves plus at least $b$  long braid moves, 
then  $\diam(G_{\pi})\geq b+c$, $\diam(C_{\pi})\geq b$ and $\diam(B_{\pi})\geq c$. 
\end{enumerate}
\end{lem}

\begin{proof}
Suppose that for any $r,r'\in R(\pi)$ there exists a path from $r$ to $r'$ with at most $c$ commutation moves plus at most $b$ long braid moves. Certainly the distance from $r$ to $r'$ is at most $b+c$. Because $d(r,r')\leq b+c$ for all $r,r'\in R(\pi)$, we know that $\diam(G_{\pi})\leq b+c$. Now consider the graph $C_{\pi}$ formed from $G_{\pi}$ by contracting along all commutation edges. Let $S$ and $S'$ be two commutation classes in $C_{\pi}$ and $v\in S$ and $v'\in S'$ be two vertices in $G_{\pi}$.  There exists a path $P$ from $v$ to $v'$ using at most $c$ commutation moves and $b$ long braid moves. The associated path in $C_{\pi}$ loses all commutation edges, because they are contracted, and keeps at most $b$ long braid edges so $d(S,S')\leq b$. Because $d(S,S')\leq b$ for all pairs of commutation classes we can conclude that $\diam(C_{\pi})\leq b$. Similarly, we can conclude that $\diam(B_{\pi})\leq c$. 

Suppose that there exists $r,r'\in R(\pi)$ so that  any path from $r$ to $r'$ has at least $c$ commutation moves plus at least $b$ long braid moves. Because $d(r,r')\geq b+c$ we can conclude that $\diam(G_{\pi})\geq b+c$. Now consider $C_{\pi}$ and the associated commutation equivalence classes $r\in S$ and $r'\in S'$. Suppose the contrary that $\diam(C_{\pi})<b$. This means that there exists a path from $S$ to $S'$ that uses less than $b$ long braid steps. Lift this path to some path $P$  in $G_{\pi}$ from $r$ to $r'$. Path $P$ has at least $b$ long braid steps, so when we contract $G_{\pi}$ to $C_{\pi}$ at least one of these long braid edges ends up contracted. Particularly, we then must have two $v,v'\in G_{\pi}$ connected by a long braid edge $e$, where $v$ and $v'$ end up in the same commutation class. This means there exists a path $Q$ from $v$ to $v'$ composed of only commutation moves. However, all cycles in $G_{\pi}$ have an even number of braid moves and an even number of commutation moves~\cite{RR13}. The cycle $Q$ together with $e$ has an odd number of braid moves, which is a contradiction. Hence, $\diam(C_{\pi})\geq b$ and similarly we can conclude that $\diam(B_{\pi})\geq c$. 
\end{proof}

Let $P$ be a poset. The \textit{Hasse diagram} of a finite poset $P$ is the graph whose vertices are the elements of $P$, whose edges are the cover relations, and such that if $x\lessdot y$ then $y$ is drawn above $x$. We say $P$ has a $\hat{1}$ if there exists an element $\hat{1}\in P$ such that $x\leq \hat{1}$ for all $x\in P$. Similarly, we say $P$ has a $\hat{0}$ if there exists an element $\hat{0}\in P$ such that $\hat{0}\leq x$ for all $x\in P$. We call the elements $\hat{1}$ and $\hat{0}$, if exist, the \textit{maximum} and the \textit{minimum} elements of $P$ respectively. A subset $C=\{x_1,x_2,\dots,x_n\}$ of $P$ is called a \textit{chain} if $x_1<x_2<\dots<x_n$. A chain is called \textit{maximal} if it is not contained in a longer chain of $P$. The \textit{length} $\ell(C)$ of a finite chain is defined by $\ell(C)=|C|-1$. If $P$ is a poset such that all maximal chains have the same length $n$, then we say that $P$ is \textit{graded of rank $n$}. In this case there is a unique \textit{rank function} $\rho: P\to\{0,1,\dots,n\}$ such that $\rho(x)=0$ if $x$ is a minimal element of $P$, and $\rho(y)=\rho(x)+1$ if $x\lessdot y$ in $P$.

A \textit{linear hyperplane} of $\mathbb{R}^d$ is a $(d-1)$-dimensional subspace $H=\{v\in\mathbb{R}^d:\alpha\cdot v=0\}$ of $\mathbb{R}^d$ where $\alpha\in\mathbb{R}^d$ is a fixed nonzero vector and $\alpha\cdot v$ is the usual dot product. For an arrangement $\mathcal{A}$ of linear hyperplanes in $\mathbb{R}^d$, the \textit{intersection poset} $L=L(\mathcal{A})=\bigsqcup_{i=0}^d L_i$ is the graded poset of all nonempty intersections of hyperplanes, including $\mathbb{R}^d$ itself, ordered by reverse inclusion. For example, we observe that the minumum $\hat{0}=L_0=\mathbb{R}^d$ element, the maximum $\hat{1}=L_d$ element, $L_1=\{H:H\in\mathcal{A}\}$, and $L_2=\{H\cap K: H\neq K, H\in \mathcal{A}, K\in\mathcal{A}\}$ is the set of all codimension-two intersection spaces, and so on. A \textit{chamber} of an arrangement $\mathcal{A}$ is a connected component of the complement $X=\mathbb{R}^d-\bigcup_{H\in\mathcal{A}}H$ of the hyperplanes. For more details on these definitions related to hyperplane arrangements see~\cite{RR13}

Let $\mathcal{A}$ be the reflection arrangement of $\mathbb{R}^{n-1}$ of type $A_{n-1}$, which is associated with the symmetric group $\fS_n$. We  identify the ambient space with the quotient of $\mathbb{R}^n=\{(x_1,\dots,x_n):x_i\in\mathbb{R}\}$ by the subspace spanned by $\{(1,1,\dots,1)^T\}$. Then the hyperplanes in $\mathcal{A}$ are $H_{ij}:=\{x_i=x_j\}$ for $1\leq i<j\leq n$. Note  that $\mathcal{A}$ is \textit{central} and \textit{essential}, meaning that $\bigcap_{H\in\mathcal{A}}H=\{\textbf{0}\}$ where $\textbf{0}$ denotes the origin of $\mathbb{R}^d$. The codimension-two intersection subspaces in $L_2$ are either $X_{ij,k\ell}:=\{x_i=x_j,x_k=x_{\ell}\}=H_{ij}\cap H_{k\ell}$ or $X_{ijk}:=\{x_i=x_j=x_k\}=H_{ij}\cap H_{jk}$. Let $\mathcal{C}$ be the set of chambers of $\mathcal{A}$. For two chambers $c,c'$, define $L_1(c,c'):=\{H\in L_1: H \text{ separates } c \text{ from } c'\}$. We define a graph $G_1$ on $\mathcal{C}$ where two chambers $c,c'\in \mathcal{C}$ are joined if $|L_1(c,c')|=1$. A \textit{gallery} from $c$ to $c'$ is a shortest path in the graph $G_1$. Fix a particular base chamber $c_0$ and let $\mathcal{R}$ denote the set of all galleries in $G_1$ from $c_0$ to $-c_0$. We define a graph $G_2$ on the set $\mathcal{R}$ where two galleries $r,r'\in\mathcal{R}$ are joined if they are separated by exactly one codimension-two intersection subspace $X$ in $L_2$. For an intersection space $X\in L_i$ for some $i$, a \textit{localized arrangement} in $\mathbb{R}^d/X$ is $\mathcal{A}_X=\{H/X:H\in\mathcal{A}, X\subseteq H\}$, which can be identified with the interval $[\mathbb{R}^d,X]$ in the intersection poset $L$. For $\pi\in\fS_n$, define $L_1(\pi):=L_1(c_0,\pi(c_0))$ and $L_2(\pi):=\{X\in L_2: \mathcal{A}_X\in L_1(\pi)\}$. We can interpret the set $L_1(\pi)$ as the usual (left) inversion set of $\pi$. We can also interpret $L_2(\pi)$ as the set $I_2(\pi)\cup I_3(\pi)$ where $I_2(\pi)$ is the set of all disjoint pairs of inversions $((i,j), (k,\ell))$ of $\pi$ and $I_3(\pi)$ is the set of all triples of inversions $((i,j), (i,k), (j,k))$ of $\pi$.  

\section{12-Inflations}
\label{sec:12-inflations}
In this section we will describe the collection of reduced words for permutations that are formed from {\it 12-inflations}, that is permutations equal to $\pi=12[\alpha,\beta]$ for some permutations $\alpha$ and $\beta$. We will also find exact recursive formulas for the diameters of the graphs $G_{\pi}$, $C_{\pi}$ and $B_{\pi}$ for 12-inflations. 

Let $u=u_1u_2\ldots u_k$ and $v=v_1v_2\ldots v_l$ be two sequences of integers. A {\it shuffle} of $u$ and $v$ is a sequence $w=w_1w_2\ldots w_{k+l}$ of integers with a subsequence 
$w_{i_1}w_{i_2}\ldots w_{i_k}$ equal to $u$ and another subsequence $w_{j_1}w_{j_2}\ldots w_{j_l}$ equal to $v$ where $\{i_1,i_2,\ldots, i_k, j_1,j_2,\ldots, j_l\}=[k+l]$. 
Let $\Shuff(u,v)$ be the collection of all shuffles of $u$ and $v$.
In order to describe the set of reduced words of $12$-inflations, $\pi=12[\alpha,\beta]$, we will be shuffling reduced words of $\alpha$ and $\beta$.   In these shuffles we will want to distinguish the letters that come from reduced words of $\alpha$ from those that come from  reduced words of $\beta$. To do this we will write reduced words of $\alpha\in\fS_a$ in the alphabet $\ul{[a-1]}=\{\ul{1},\ul{2},\ldots, \ul{a-1}\}$ and reduced words of $\beta\in\fS_b$ in the alphabet $\ol{[b-1]}=\{\ol{1},\ol{2},\ldots, \ol{b-1}\}$. Denote these sets  $\ul{R}(\alpha)$ and $\ol{R}(\beta)$ respectively. 

\begin{exmp} For $\alpha=21$ and $\beta=231$, we observe that 
$\ul{R}(\alpha)=\{\ul{1}\}$, $\ol{R}(\beta)=\{\ol{12}\}$ and $\Shuff(\ul{1},\ol{1}\ol{2})=\{\ul{1}\ol{1}\ol{2},\ol{1}\ul{1}\ol{2},\ol{1}\ol{2}\ul{1}\}$. 
\end{exmp}

Let $\pi=12[\alpha,\beta]$ and 
$$U_{\alpha,\beta} =  \bigcup_{u\in \ul{R}(\alpha),{v}\in \ol{R}(\beta)} \Shuff(u,{v}).$$
We will define a graph $G_{\alpha,\beta}$ with vertices $U_{\alpha,\beta}$ and edges formed from the following relations:
\begin{enumerate}
\item Commutation moves, which come from exchanging adjacent letters in the following cases. 
\begin{enumerate}
\item $\ul{jk}$ if $|j-k|>1$, 
\item $\ol{jk}$ if $|j-k|>1$ 
\item $\ul{j}\ol{k}$ for any $j$ and $k$
\end{enumerate}
\item Long braid moves, which comes from exchanging the following  occurrences on consecutive indices. 
\begin{enumerate}
\item $\ul{j(j + 1)j}$  and $\ul{(j + 1)j(j + 1)}$ 
\item $\ol{j(j + 1)j}$ and $\ol{(j + 1)j(j + 1)}$
\end{enumerate}
\end{enumerate}
We will call moves like 1(a) and 2(a) {\it $\alpha$-moves}, moves like 1(b) and 2(b) {\it $\beta$-moves} and moves in 2(c) {\it shift-moves}. 
We will  show that $U_{\alpha,\beta}$ is in bijection with $R(\pi)$ by showing that the graph $G_{\alpha,\beta}$ is isomorphic to $G_{\pi}$. 

First we will define the map between the vertices, $\eta:U_{\alpha,\beta}\rightarrow [a+b-1]^{\ell(\pi)}$, as follows, where $[k]^m$ is the set of all words length $m$ with letters in $[k]$. Let $w\in U_{\alpha,\beta}$.  Then $\eta(w)=r$ is defined by 
$$
r_i=\left\{
\begin{array}{ll}
j&\text{ if }w_i=\ul{j}\\
j+a&\text{ if }w_i=\ol{j}.
\end{array}
\right.
$$
Notice that all outputs of $\eta$ are in  $[a+b-1]^{\ell(\pi)}$. See Figure~\ref{fig:12-graph} for an example of $G_{\pi}$ for a 12-inflation. 

\begin{exmp}
For $\alpha=21$ and $\beta=321$, we see that
$w=\ol{1}\ul{1}\ol{21}\in U_{\alpha,\beta}$ and $\eta(w)=3143$. 
\end{exmp}

We claim that $\eta$ is a bijection between $U_{\alpha,\beta}$ and $R(\pi)$, and further that $G_{\alpha,\beta}$ is isomorphic to $G_{\pi}$. To prove this, we first state and prove the following lemma.

\begin{lem}
\label{lem:12_edge_iff}
Let $\pi=12[\alpha,\beta]$, $|\alpha|=a$ and $|\beta|=b$. The map $\eta$ is injective. Also, if $w\in U_{\alpha,\beta}$ and $\eta(w)=r$ then we can describe exactly the commutation and braid moves of $r$ with conditions on $w$. 
\begin{enumerate}
\item We can perform a commutation move on $r_ir_{i+1}$ in $r$ if and only if $w_iw_{i+1}$ equals
\begin{enumerate}
\item $\ul{jk}$ for some $|j-k|>1$,
\item $\ol{jk}$ for some $|j-k|>1$ or
\item $\ul{j}\ol{k}$ or $\ol{k}\ul{j}$ for any $j$ and $k$. 
\end{enumerate}
\item We can perform a long braid move on $r_ir_{i+1}r_{i+2}$ in $r$ if and only if $w_iw_{i+1}w_{i+2}$ equals
\begin{enumerate}
\item $\ul{j(j+1)j}$ or $\ul{j(j-1)j}$ for any $j$ or
\item $\ol{j(j+1)j}$ or $\ol{j(j-1)j}$ for any $j$.
\end{enumerate}
\end{enumerate}
\end{lem}
\begin{proof}

Let $w\in U_{\alpha,\beta}$ and $\eta(w)=r$. To show parts 1 and 2 it will suffice to show the following two notes. The first note is that $w_i-w_j=r_i-r_j$ if $w_i$ and $w_{i+1}$ are both in $\ul{[a-1]}$ or are both in $\ol{[b-1]}$. The second note is that $|r_i-r_j|>1$ if one of $w_i$ and $w_{i+1}$ is in $\ul{[a-1]}$ and the other is in $\ol{[b-1]}$. 

To show the first note, consider $w_iw_{i+1}=\ul{jk}$ for some $j$ and $k$. Then
$r_ir_{i+1}=jk.$ If $w_iw_{i+1}=\ol{jk}$ for some $j$ and $k$, then
$r_ir_{i+1}=(j+a)(k+a).$ This justifies the first note. 
To prove the second note consider the case where $w_iw_{i+1}=\ul{j}\ol{k}$ for any $j$ and $k$, then $r_ir_{i+1}=j(k+a).$ Since $j<a$ and $k+a>a$ we  have that $|r_i-r_j|>1$. Similarly if $w_iw_{i+1}=\ol{j}\ul{k}$ for any $j$ and $k$, then $|r_i-r_j|>1$. This proves part 1 and 2.

Next, we will show that $\eta$ is injective. Suppose that $\eta(w)=\eta(w')=r$ and $w\neq w'$. This means that there exists an $i$ with $w_i\neq w'_i$. If $w_i$ and $w'_i$ are both letters in $\ul{[a-1]}$, then because $w_i\neq w'_i$ the output of at index $i$ must be different by our definition of $\eta$. This is the same if we had supposed that $w_i$ and $w'_i$ are both letters in $\ol{[b-1]}$. The last case is if $w_i$ and $w'_i$ are in separate sets, one in $\ul{[a-1]}$ and the other in  $\ol{[b-1]}$. In this case the output at index $i$ must also be different since letters in $\ul{[a-1]}$ map to numbers less than $a$, and letters in $\ol{[b-1]}$ map to numbers more than $a$. Hence, $\eta$ must be injective. 
\end{proof}

\begin{thm}
\label{thm:12_graph_iso}
Let $\pi=12[\alpha,\beta]$. The map $\eta$ is a bijection between $U_{\alpha,\beta}$ and $R(\pi)$ and  $G_{\alpha,\beta}$ is isomorphic to $G_{\pi}$. 
\end{thm}

\begin{proof}
We have already shown that $\eta$ is injective in Lemma~\ref{lem:12_edge_iff}. If we form a graph on the image of $\eta$ by connecting vertices according to possible commutation and long braid moves, Lemma~\ref{lem:12_edge_iff} proves that we get a graph that is isomorphic to $G_{\alpha,\beta}$. Let $I$ denote the graph that is formed from the image of $\eta$. Now, we will only have to show two things to prove that $G_{\alpha,\beta}$ is isomorphic to $G_{\pi}$, so also $\eta$ is a bijection between $U_{\alpha,\beta}$ and $R(\pi)$. We will first show that there exists a specific $\tilde w$ that maps to $\tilde r$ that is in $R(\pi)$. Because $G_{\pi}$ is a connected graph that can be generated by one single vertex $\tilde r\in R(\pi)$ by using commutation and braid moves, we can conclude that one connected component of $I$ is isomorphic to $G_{\pi}$. The second thing we will show is that $G_{\alpha,\beta}$ is connected, which proves that $G_{\alpha,\beta}$ is isomorphic to $G_{\pi}$. 

First, let us describe the specific $\tilde w\in U_{\alpha,\beta}$ where  $\eta(\tilde w)=\tilde r$  is in $R(\pi)$. Pick some $\tilde u\in R(\alpha)$ and $\tilde v\in R(\beta)$. Certainly the concatenation $\tilde w=\ul{\tilde u_1\tilde u_2\ldots \tilde u_{\ell(\alpha)}}\ol{\tilde v_1\tilde v_2\ldots \tilde v_{\ell(\beta)}}$ is in  $U_{\alpha,\beta}$. Further, $\eta(\tilde w) = \tilde r=\tilde u_1\tilde u_2\ldots \tilde u_{\ell(\alpha)}(\tilde v_1+a)(\tilde v_2+a)\ldots (\tilde v_{\ell(\beta)}+a)$ is in $R(\pi)$. 

Next we will show that $G_{\alpha,\beta}$ is connected by describing a path from any $w\in U_{\alpha,\beta}$  to $\tilde w$. Since $w\in U_{\alpha,\beta}$ we know that $w$ is formed from a shuffle of some $u\in \ul{R}(\alpha)$ and $v\in \ol{R}(\beta)$. Because we can perform commutation moves on $\ul{j}\ol{k}$ and $\ol{k}\ul{j}$ for any $j$ and $k$ we know that there is a path from $w$  to the concatenation $uv$ via commutation moves. Because we can perform commutation and long braid moves on the letters $\ul{[a-1]}$ or the letters $\ol{[b-1]}$ and because $G_{\alpha}$ and $G_{\beta}$ are connected there exist a sequences of commutation and long braid moves to transform $u$ to $\tilde u$ and $v$ to $\tilde v$. Hence, there must be a path from $uv$  to $\tilde w$ in $G_{\alpha,\beta}$, which completes the proof.
\end{proof}

We are now ready to prove an exact recursive formula for the diameters of $G_{\pi}$,  $C_{\pi}$ and  $B_{\pi}$. It will be helpful to define a statistic on words in $U_{\alpha,\beta}$. Given $w\in U_{\alpha,\beta}$ let $\shift(w)$ count the number of pairs of indices $i<i'$ such that $w_i\in \ul{[a-1]}$ and $w_{i'}\in \ol{[b-1]}$. 

\begin{exmp}
Given $\alpha=21$, $\beta=321$ and  
$w=\ol{1}\ul{1}\ol{21}\in U_{\alpha,\beta}$ we have that $\shift(w)=2$. 
\end{exmp}

\begin{thm}
\label{thm:12_formula}
Let $\pi=12[\alpha,\beta]$. 
\begin{enumerate}[(i)]
\item $\diam (G_{\pi})=\diam( G_{\alpha}) + \diam (G_{\beta}) + \ell(\alpha)\ell(\beta)$
\item $\diam (C_{\pi})=\diam (C_{\alpha}) + \diam( C_{\beta})$
\item $\diam( B_{\pi})=\diam( B_{\alpha} )+ \diam (B_{\beta} )+ \ell(\alpha)\ell(\beta)$
\end{enumerate}
\end{thm}

\begin{proof}
By Theorem~\ref{thm:12_graph_iso} we know that $G_{\pi}$ is isomorphic to $G_{\alpha,\beta}$, so it suffices to prove this theorem on the graph $G_{\alpha,\beta}$. We will need to consider  specific subgraphs of $G_{\alpha,\beta}$. The first is on the vertices $\{uv:u\in\ul{R}(\alpha), v\in\ol{R}(\beta)\}$, which we will call $G_1$. The second is on the vertices $\{vu:u\in\ul{R}(\alpha), v\in\ol{R}(\beta)\}$, which we will call $G_2$. Note that we can transform any two vertices in $G_1$ into each other by using at most $\diam( G_{\alpha})$ $\alpha$-moves plus at most $\diam( G_{\beta})$ $\beta$-moves. Similarly, we can transform any two vertices in $G_2$ into each other by using at most $\diam( G_{\alpha})$ $\alpha$-moves plus at most $\diam( G_{\beta})$ $\beta$-moves. 

We claim that given any $w\in U_{\alpha,\beta}$ that any path from $w$ to any vertex in $G_2$ takes at least $\shift(w)$ shift-moves, which are commutation moves. 
Note that $G_2$ contains exactly those vertices $w$ with $\shift(w)=0$. Also note that $\alpha$-moves and $\beta$-moves do not change $\shift(w)$, but shift-moves change $\shift(w)$ exactly by one. This means any path from $w$ to $G_2$ requires at least $\shift(w)$ shift-moves. 
Finally, note that there exists a path from $w$ to a vertex in $G_2$ that takes exactly $\shift(w)$ shift-moves. 
Because $G_1$ contains exactly those vertices $w$ with $\shift(w)=\ell(\alpha)\ell(\beta)$ we can similarly conclude that any path from $w$ to any vertex in $G_1$ takes at least $\ell(\alpha)\ell(\beta)-\shift(w)$ shift-moves. 
Also, note that there does exist a path from $w$ to a vertex in $G_1$ that uses exactly $\ell(\alpha)\ell(\beta)-\shift(w)$ shift-moves. 

Next we will prove an upper bound for the diameter of $G_{\pi}$. Consider any pair of $w$ and $w'$. We will show that there exists a path between  $w$ and $w'$ of length at most 
$\diam( G_{\alpha}) + \diam (G_{\beta}) + \ell(\alpha)\ell(\beta)$. To do this we will describe two paths from $w$ to $w'$. Let the first path, $P_1$, start at  $w$ that then takes $\shift(w)$ shift-steps to get to a vertex $h_2$ in $G_2$. The path $P_1$ ends at $w'$ with the $\shift(w')$ shift-moves to get you from some vertex $h_2'$ in  $G_2$ to $w'$. We complete $P_1$ by connecting $h_2$ to $h_2'$ with at most $\diam( G_{\alpha})$ $\alpha$-moves plus at most $\diam( G_{\beta})$ $\beta$-moves. The length of $P_1$ is at most $\shift(w)+\shift(w')$ shift-moves plus $\diam( G_{\alpha})$ $\alpha$-moves plus $\diam( G_{\beta})$ $\beta$-moves. We can similarly connect $w$ to $w'$ with another path, $P_2$, through $G_1$ that will have a length at most $2\ell(\alpha)\ell(\beta)-\shift(w)-\shift(w')$ shift-moves plus $\diam( G_{\alpha})$ $\alpha$-moves plus $\diam( G_{\beta})$ $\beta$-moves. All together the cycle formed by combining the paths $P_1$ and $P_2$ has length at most 
$$2\ell(\alpha)\ell(\beta)+2\diam( G_{\alpha})+2\diam( G_{\beta}).$$
This implies either $P_1$ or $P_2$ has length at most 
$\ell(\alpha)\ell(\beta)+\diam( G_{\alpha})+\diam( G_{\beta})$ 
proving that $d(w,w')\leq \ell(\alpha)\ell(\beta)+\diam( G_{\alpha})+\diam( G_{\beta})$. 
Hence, $\diam(G_{\pi})\leq \ell(\alpha)\ell(\beta)+\diam( G_{\alpha})+\diam( G_{\beta})$. 

Now we will prove a lower bound for the diameter of $G_{\pi}$. Pick two $u,u'\in R(\alpha)$ with $d(u,u')=\diam(G_{\alpha})$ and  two $v,v'\in R(\beta)$ with $d(v,v')=\diam(G_{\beta})$. Consider $w=\ul{u}\ol{v}$ and $w'=\ol{v'}\ul{u'}$ and any path $P$ between them. We have shown that this path will require at least $\ell(\alpha)\ell(\beta)$ shift-moves. We can also conclude that this path requires at least $\diam(G_{\alpha})$ $\alpha$-moves, otherwise we could project our path onto $G_{\alpha}$ and have a path from $u$ to $u'$ that is shorter than $\diam(G_{\alpha})$. Similarly, we will have at least $\diam(G_{\beta})$ $\beta$-moves. Hence, the path $P$ has length at least $\ell(\alpha)\ell(\beta)+\diam(G_{\alpha})+\diam(G_{\beta})$, proving that $\diam(G_{\pi})\geq \ell(\alpha)\ell(\beta)+\diam(G_{\alpha})+\diam(G_{\beta})$. Together with our upper bound we have proven our recursion on  the diameter of $G_{\pi}$. 

Using the ideas above we can show that for any $w,w'\in U_{\alpha,\beta}$ that we can find a path that uses  at most $\diam(C_{\alpha})$ long braid $\alpha$-moves, at most $\diam(C_{\beta})$ long braid $\beta$-moves and otherwise just uses commutation moves of different types. By Lemma~\ref{lem:diam_bounds} this proves that $\diam(C_{\pi})\leq \diam(C_{\alpha})+\diam(C_{\beta})$. Also, using similar ideas as above we can construct $w,w'\in U_{\alpha,\beta}$ where all paths from $w$ to $w'$ require at least $\diam(C_{\alpha})$ long braid $\alpha$-moves, at least $\diam(C_{\beta})$ long  braid $\beta$-moves and otherwise just uses commutation moves of different types. By Lemma~\ref{lem:diam_bounds} this proves that $\diam(C_{\pi})\geq \diam(C_{\alpha})+\diam(C_{\beta})$ and our recursion for $\diam(C_{\pi})$ is proven. 

Again, using the ideas above we can show that for any $w,w'\in U_{\alpha,\beta}$ that we can find a path that uses at most $\ell(\alpha)\ell(\beta)$ shift-moves, at most $\diam(B_{\alpha})$ commutation $\alpha$-moves, at most $\diam(B_{\beta})$ commutation $\beta$-moves and otherwise just uses long braid moves. By Lemma~\ref{lem:diam_bounds} this proves that $\diam(B_{\pi})\leq \diam(B_{\alpha})+\diam(B_{\beta})+\ell(\alpha)\ell(\beta)$. Also, using similar ideas as above we can construct $w,w'\in U_{\alpha,\beta}$ where all paths from $w$ to $w'$ require at least $\ell(\alpha)\ell(\beta)$ shift-moves, at least $\diam(B_{\alpha})$ commutation $\alpha$-moves, at least $\diam(B_{\beta})$ commutation $\beta$-moves and otherwise just uses braid moves. By Lemma~\ref{lem:diam_bounds} this proves that $\diam(B_{\pi})\geq \diam(B_{\alpha})+\diam(B_{\beta})+\ell(\alpha)\ell(\beta)$ and our recursion for $\diam(B_{\pi})$ is proven. 
\end{proof}

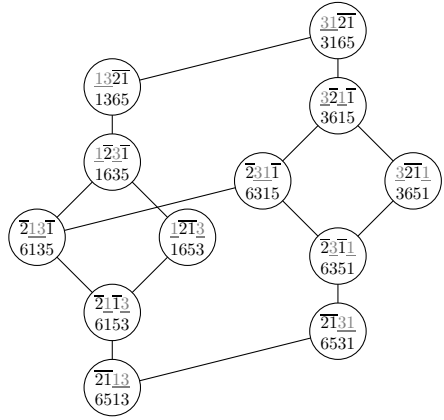
\begin{figure}
\begin{center}
\begin{tikzpicture}[every text node part/.style={align=center}, scale = .5,every node/.style={scale=0.6}]

\coordinate (A) at (0,0);
\coordinate (B) at (0,2);
\coordinate (C) at (2,4);
\coordinate (D) at (-2,4);
\coordinate (E) at (0,6);
\coordinate (F) at (0,8);
\coordinate (A1) at (6,1.5);
\coordinate (B1) at (6,3.5);
\coordinate (C1) at (8,5.5);
\coordinate (D1) at (4,5.5);
\coordinate (E1) at (6,7.5);
\coordinate (F1) at (6,9.5);
\draw (A) -- (B) -- (C) -- (E) -- (F);
\draw (B) -- (D) -- (E);
\draw (A1) -- (B1) -- (C1) -- (E1) -- (F1);
\draw (B1) -- (D1) -- (E1);
\draw (A) -- (A1);
\draw (D) -- (D1);
\draw (F) -- (F1);

\draw[draw = black, fill = white] (A) circle (.75);
\draw (A) node {$\ol{21}\tg{\ul{13}}$ \\ $6513$};
\draw[draw = black, fill = white] (B) circle (.75);
\draw (B) node {$\ol{2}\tg{\ul{1}}\ol{1}\tg{\ul{3}}$ \\ $6153$};
\draw[draw = black, fill = white] (C) circle (.75);
\draw (C) node {$\tg{\ul{1}}\ol{2}\ol{1}\tg{\ul{3}}$ \\ $1653$};
\draw[draw = black, fill = white] (D) circle (.75);
\draw (D) node {$\ol{2}\tg{\ul{1}}\tg{\ul{3}}\ol{1}$ \\ $6135$};
\draw[draw = black, fill = white] (E) circle (.75);
\draw (E) node {$\tg{\ul{1}}\ol{2}\tg{\ul{3}}\ol{1}$ \\ $1635$};
\draw[draw = black, fill = white] (F) circle (.75);
\draw (F) node {$\tg{\ul{13}}\ol{21}$ \\ $1365$};

\draw[draw = black, fill = white] (A1) circle (.75);
\draw (A1) node {$\ol{21}\tg{\ul{31}}$ \\ $6531$};
\draw[draw = black, fill = white] (B1) circle (.75);
\draw (B1) node {$\ol{2}\tg{\ul{3}}\ol{1}\tg{\ul{1}}$ \\ $6351$};
\draw[draw = black, fill = white] (C1) circle (.75);
\draw (C1) node {$\tg{\ul{3}}\ol{2}\ol{1}\tg{\ul{1}}$  \\ $3651$};
\draw[draw = black, fill = white] (D1) circle (.75);
\draw (D1) node {$\ol{2}\tg{\ul{3}}\tg{\ul{1}}\ol{1}$  \\ $6315$};
\draw[draw = black, fill = white] (E1) circle (.75);
\draw (E1) node {$\tg{\ul{3}}\ol{2}\tg{\ul{1}}\ol{1}$ \\ $3615$};
\draw[draw = black, fill = white] (F1) circle (.75);
\draw (F1) node {$\tg{\ul{31}}\ol{21}$ \\ $3165$};
\end{tikzpicture}
\end{center}
\caption{This is the graph $G_{\pi}$ of the 12-inflation $\pi=12[2143,312]=2143756$ with vertices given in both $U_{2143,312}$ and $R(\pi)$.}
\label{fig:12-graph}
\end{figure}

\section{21-Inflations}
\label{sec:21-inflations}
In this section we describe the collection of reduced words for permutations that are formed from {\it 21-inflations}, that is a permutation equal to $\pi=21[\alpha,\beta]$ for some permutations $\alpha$ and $\beta$. The graphs of 21-inflations are more complex than 12-inflations, so we will only find an exact recursive formula for the diameters of $G_{\pi}$,  $C_{\pi}$ and  $B_{\pi}$ for 21-inflations $21[\alpha,1]$ and $21[1,\beta]$. Because the longest permutation $\delta_n=n\ldots 21$ is a 21-inflation of the form $21[\alpha,1]$ we have an exact recursive formula for the diameters of $G_{\delta_n}$, $C_{\delta_n}$ and $B_{\delta_n}$.  For  other 21-inflation of the form $\pi=21[\alpha,\iota_k]$ we provide recursive upper and lower  bounds for the diameters of $G_{\pi}$, $C_{\pi}$ and $B_{\pi}$.  

\subsection{Reduced Words of 21-Inflations}
\label{subsec:21-inflation_ecoding}
We will first describe a set that is in bijection with reduced words in $R(\pi)$ where $\pi=21[\alpha,\beta]$, $|\alpha|=a$ and $|\beta|=b$. This set will contain certain shuffles of $R(\alpha)$, $R(\beta)$ and ballot sequences.

Let $x=x_1x_2\ldots x_n$ be a sequence of positive integers. Define $N_j(x)$ to be the number of $j$'s in $x$. We call $x$ a {\it ballot sequence} if for all positive integers $j<k$ we have that $N_j(x_1x_2\ldots x_m)\geq N_k(x_1x_2\ldots x_m)$ for all $m\in [n]$. Since we will often be referring to prefixes $x_1x_2\ldots x_m$ of $x=x_1x_2\ldots x_n$, $m\leq n$, denote $x^{(m)}=x_1x_2\ldots x_m$. 
Let $\Ballot_{a,b}$ be the collection of all ballot sequences that are rearrangements of the multiset $\{1^b, 2^b, \ldots, a^b\}$ where $i^b$ means there are $b$ copies of $i$. A {\it reverse ballot sequence} is a sequence of positive integers at most $L$, for some $L$, where for all $j<k\leq L$ we have that  $N_j(x^{(m)})\leq N_k(x^{(m)})$ for all $m\in [n]$. Let $\RBallot_{a,b}$ be the collection of all reverse ballot sequences on the same multiset $\{1^b, 2^b, \ldots, a^b\}$. There is a natural bijection that will be useful to us, 
$$f:\Ballot_{a,b}\rightarrow \RBallot_{b,a}.$$
We map $x\in \Ballot_{a,b}$ to $y\in \RBallot_{b,a}$ by
$$y_i = b-N_j(x^{(i-1)})\text{ if }x_i=j.$$ 
Later we will have words $w=w_1w_2\ldots w_n$ with letters from $[a+b-1]\cup\ul{[a-1]}\cup\ol{[b-1]}$. We can still define $N_j(w)$ to be the number of $j$'s in $w$ where we do not count $\ul{j}$ or $\ol{j}$. Also, if the subsequence of $w=w_1w_2\ldots w_n$ formed from only letters in $[a+b]$ is a ballot sequence in $\Ballot_{a,b}$ then we can define $f(w)$ similarly and only apply the map to the subsequence that is the ballot sequence. 

\begin{exmp} The sequence $112323$ is a ballot sequence in $\Ballot_{3,2}$ and $212211$ is a reverse ballot sequence in $\RBallot_{2,3}$. We have $f(112323)=212211$ and $f(11\ol{1}23\ul{2}2\ul{1}3)=21\ol{1}22\ul{2}1\ul{1}1$.
\end{exmp}

We are now ready to start defining the set that is in bijection with $R(\pi)$ where $\pi=21[\alpha,\beta]$. Given $x\in \Ballot_{a,b}$, $u\in \ul{R}(\alpha)$ and $v\in \ol{R}(\beta)$ define $\ol{\Shuff}(x,u,v)$ to the the collection of all $w\in {\Shuff}(x,u,{v})$ such that 
\begin{enumerate}
\item if $w_i=\ul{j}$ then $N_j(w^{(i-1)})=N_{j+1}(w^{(i-1)})$ and
\item if $w_i=\ol{j}$ then $N_j(f(w)^{(i-1)})=N_{j+1}(f(w)^{(i-1)})$
\end{enumerate}

\begin{exmp}Consider $x=112323$ in $\Ballot_{3,2}$, $u=\ul{21}$ in $\ul{R}(312)$ and ${v}=\ol{1}$ in $ \ol{R}(21)$. Then, $\ol{\Shuff}(x,u,{v})$  contains elements like $\ul{21}\ol{1}112323$ and $1\ul{2}1\ol{1}232\ul{1}3$. 
\end{exmp}

We are going to find that these shuffles in $\ol{\Shuff}(x,u,v)$ encode reduced words of $\pi$ so let 
$$V_{\alpha,\beta}=\underset{x\in \Ballot_{a,b}}{\bigcup_{u\in \ul{R}(\alpha),{v}\in \ol{R}(\beta)}} \ol{\Shuff}(x,u,{v}).$$
We will define a graph $H_{\alpha,\beta}$ with vertices $V_{\alpha,\beta}$ and edges formed from the following relations:
\begin{enumerate}
\item Commutation moves, which come from exchanging adjacent letters in the following cases.
\begin{enumerate}
\item $w_iw_{i+1}=\ul{pq}$ if $|p-q|>1$
\item $w_iw_{i+1}=\ol{pq}$ if $|p-q|>1$
\item $w_iw_{i+1}=pq$ if either $p>q$ or $p<q$ and $N_p(w^{(i-1)})>N_q(w^{(i-1)})$
\item  $w_iw_{i+1}=\ul{p}\ol{q}$ or $w_iw_{i+1}=\ol{q}\ul{p}$ for any $p$ and $q$
\item $w_iw_{i+1}=\ul{p}q$ or $w_iw_{i+1}=q\ul{p}$ if $|p-q|>1$ or $p>q$ 
\item $z_iz_{i+1}=\ol{p}q$ or $z_iz_{i+1}=q\ol{p}$ if $|p-q|>1$ or $q<p$ where $f(w)=z$
\end{enumerate}
\item Long braid moves, which comes from exchanging the following occurrences on consecutive
indices.
\begin{enumerate}
\item in $w$ we exchange $\ul{p(p+ 1)p}$ and $\ul{(p+1)p(p+1)}$
\item in $w$ we exchange $\ol{p(p+ 1)p}$ and $\ol{(p+1)p(p+1)}$
\item  in $w$ we exchange $\ul{p}p(p+1)$ and  $p(p+1)\ul{p}$
\item in $z$ we exchange $\ol{p}(p+1)p$ and $(p+1)p \ol{p}$ where $f(w)=z$
\end{enumerate}
\end{enumerate}
There is some intuition behind the choice of moves given above. We allow all possible commutation moves as long as the commutation move doesn't bring us outside the set $V_{\alpha,\beta}$. We allow all possible long braid moves that occur purely on the letters in $\ul{[a-1]}$ or in $\ol{[b-1]}$ that keep us in the set $V_{\alpha,\beta}$, and we have two new kinds of long braid moves, 2(c) and 2(d), that also keep us in the set $V_{\alpha,\beta}$.  See Figure~\ref{fig:21-graph} for an example. 

We can now define the map that will be a bijection between $R(\pi)$ for $\pi=21[\alpha,\beta]$ and $V_{\alpha,\beta}$,
$$\psi:V_{\alpha,\beta}\rightarrow [a+b-1]^{\ell(\pi)}.$$
For now, it will not be clear why the image of $\psi$ is actually $R(\pi)$. This will become apparent later when we show $H_{\alpha,\beta}$ is isomorphic to $G_{\pi}$.  For now it will suffice to know that the outputs are in $[a+b-1]^{\ell(\pi)}$. 
Given $w\in \ol{\Shuff}(x,u,v)$, and $f(w)=z$ since we will need $z$ to compute $\psi(w)$ in some cases. We define $\psi(w)=r$ by 
$$r_i=\left\{
\begin{array}{ll}
j+b-N_j(w^{(i-1)})-1=k+N_k(z^{(i-1)})&\text{if } w_i={j} \text{ or } z_i=k\\
j+b-N_j(w^{(i-1)})&\text{if } w_i=\ul{j}\\
j+N_j(z^{(i-1)})&\text{if } w_i=\ol{j}.\\
\end{array}
\right.
$$

\begin{exmp}Given $\alpha=312$ and $\beta=21$ we have 
$\psi(\ul{21}\ol{1}112323)=431 213423$ and 
$\psi(1\ul{2}1\ol{1}232\ul{1}3)=241234213$. 
\end{exmp}

\begin{figure}
\label{fig:21-graph}
\begin{center}
\begin{tikzpicture}[every text node part/.style={align=center}, scale = .5,every node/.style={scale=0.6}]

\coordinate (A) at (0,2);
\coordinate (B) at (3,2);
\coordinate (C) at (3,0);
\coordinate (D) at (6,0);
\coordinate (E) at (6,2);

\coordinate (F) at (9,2);
\coordinate (G) at (9,0);

\coordinate (H) at (12,0);
\coordinate (I) at (12,2);

\coordinate (J) at (15,0);
\coordinate (K) at (15,2);

\coordinate (L) at (18,0);
\coordinate (M) at (18,2);

\coordinate (N) at (21,0);

\draw (A) --(B) -- (C) --(D) --(E) --(B);
\draw[double] (E)--(F);
\draw[double] (D)--(G);
\draw (F)--(G);
\draw[double] (H)--(G);
\draw (H)--(I);
\draw[double] (H)--(J);
\draw[double] (I)--(K);
\draw (J)--(K)--(M)--(L)--(J); 
\draw (L)--(N);

\draw[draw = black, fill = white] (A) ellipse (1.25 and .75);
\draw (A) node {$111222\tg{\ul{1}}$ \\ $3214321$};
\draw[draw = black, fill = white] (B) ellipse (1.25 and .75);
\draw (B) node {$112122\tg{\ul{1}}$ \\ $3241321$};
\draw[draw = black, fill = white] (C) ellipse (1.25 and .75);
\draw (C) node {$121122\tg{\ul{1}}$ \\ $3421321$};
\draw[draw = black, fill = white] (E) ellipse (1.25 and .75);
\draw (E) node {$112212\tg{\ul{1}}$ \\ $3243121$};
\draw[draw = black, fill = white] (D) ellipse (1.25 and .75);
\draw (D) node {$121212\tg{\ul{1}}$ \\ $3423121$};

\draw[draw = black, fill = white] (F) ellipse (1.25 and .75);
\draw (F) node {$1122\tg{\ul{1}}12$ \\ $3243212$};
\draw[draw = black, fill = white] (G) ellipse (1.25 and .75);
\draw (G) node {$1212\tg{\ul{1}}12$ \\ $3423212$};

\draw[draw = black, fill = white] (H) ellipse (1.25 and .75);
\draw (H) node {$12\tg{\ul{1}}1212$ \\ $3432312$};
\draw[draw = black, fill = white] (I) ellipse (1.25 and .75);
\draw (I) node {$12\tg{\ul{1}}1122$ \\ $3432132$};

\draw[draw = black, fill = white] (J) ellipse (1.25 and .75);
\draw (J) node {$\tg{\ul{1}}121212$ \\ $4342312$};
\draw[draw = black, fill = white] (K) ellipse (1.25 and .75);
\draw (K) node {$\tg{\ul{1}}121122$ \\ $4342132$};

\draw[draw = black, fill = white] (L) ellipse (1.25 and .75);
\draw (L) node {$\tg{\ul{1}}112212$ \\ $4324132$};
\draw[draw = black, fill = white] (M) ellipse (1.25 and .75);
\draw (M) node {$\tg{\ul{1}}112122$ \\ $4324312$};

\draw[draw = black, fill = white] (N) ellipse (1.25 and .75);
\draw (N) node {$\tg{\ul{1}}111222$ \\ $4321432$};
\end{tikzpicture}
\end{center}
\caption{This is the graph $G_{\pi}$ of the 21-inflation $\pi=21[21,123]=54123$ with vertices given in both $V_{21,123}$ and $R(\pi)$.}
\end{figure}
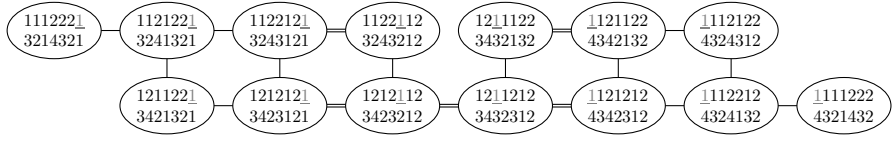

We will show that the commutation moves and braid moves we have defined on $V_{\alpha,\beta}$ match with the possible commutation and braid moves of their associated outputs. 
\begin{lem}
\label{lem:21-edge_iff}
The map $\psi:V_{\alpha,\beta}\rightarrow [a+b-1]^{\ell(\pi)}$  is injective. Also, if $\psi(w)=r$ with $f(w)=z$ then we can describe exactly the commutation and braid moves of $r$ with conditions on $w$ or $f(w)=z$.
\begin{enumerate}
\item We can perform a commutation move on $r_ir_{i+1}$ in $r$ if and only if
\begin{enumerate}
\item $w_iw_{i+1}=\ul{pq}$ if $|p-q|>1$
\item $w_iw_{i+1}=\ol{pq}$ if $|p-q|>1$
\item $w_iw_{i+1}=pq$ if either $p>q$ or $p<q$ and $N_p(w^{(i-1)})>N_q(w^{(i-1)})$
\item  $w_iw_{i+1}=\ul{p}\ol{q}$ or $w_iw_{i+1}=\ol{q}\ul{p}$ for any $p$ and $q$
\item $w_iw_{i+1}=\ul{p}q$ or $w_iw_{i+1}=q\ul{p}$ if $|p-q|>1$ or $p>q$
\item $z_iz_{i+1}=\ol{p}q$ or $z_iz_{i+1}=q\ol{p}$ if $|p-q|>1$ or $q<p$. 
\end{enumerate}
\item We can perform a long braid move on $r_ir_{i+1}r_{i+2}$ in $r$ if and only if
\begin{enumerate}
\item $w_iw_{i+1}w_{i+2}=\ul{p(p\pm 1)p}$
\item  $w_iw_{i+1}w_{i+2}=\ol{p(p\pm 1)p}$
\item either $w_iw_{i+1}w_{i+2}$ equals $\ul{p}p(p+1)$ or $p(p+1)\ul{p}$
\item either $z_iz_{i+1}z_{i+2}$ equals  $\ol{p}(p+1)p$ or $(p+1)p \ol{p}$
\end{enumerate}
\end{enumerate}
\end{lem}

\begin{proof}
Let $w\in V_{\alpha,\beta}$, $\psi(w)=r$ and $f(w)=z$. In the first part of this proof we will investigate the output $r_ir_{i+1}$ for different choices of $w_iw_{i+1}$. The cases are recorded in Table~\ref{tab:21-cases} for reference. Looking at this table we can confirm part 1 of this lemma. 

{\bf Case 1:} We have  $w_iw_{i+1}=\ul{jk}$ for some $j$ and $k$. This means 
$$r_ir_{i+1}=(j+b-N_j(w^{(i-1)}))(k+b-N_k(w^{(i)})).$$ 
Notice that $j\neq k$ because $jk$ is an adjacent subsequence for some reduced word of $\alpha$. We know that if $j<k$ then $N_j(w^{(i)})\geq N_k(w^{(i)})$ and if $j>k$ then $N_j(w^{(i)})\leq N_k(w^{(i)})$.   Because $N_j(w^{(i-1)})=N_j(w^{(i)})$ we can conclude that if $j<k$ then $r_i<r_{i+1}$ and if $j>k$ then $r_i>r_{i+1}$. Further because $N_j(w^{(i-1)})=N_{j+1}(w^{(i-1)})$ and $N_k(w^{(i)})=N_{k+1}(w^{(i)})$ we can conclude that $|j-k|=1$ if and only if $|r_i-r_{i+1}|=1$. 

{\bf Case 2:} We have  $w_iw_{i+1}=\ol{jk}$ for some $j$ and $k$. This means 
$$r_ir_{i+1}=(j+N_j(f(w)^{(i-1)}))(k+N_k(f(w)^{(i)})).$$ 
Notice that $j\neq k$ because $jk$ is an adjacent subsequence for some reduced word of $\beta$. We know that if $j<k$ then $N_j(f(w)^{(i)})\leq N_k(f(w)^{(i)})$ and if $j>k$ then $N_j(f(w)^{(i)})\geq N_k(f(w)^{(i)})$.   Because $N_j(f(w)^{(i-1)})=N_j(f(w)^{(i)})$ we can conclude that if $j<k$ then $r_i<r_{i+1}$ and if $j>k$ then $r_i>r_{i+1}$. Further because $N_j(f(w)^{(i-1)})=N_{j+1}(f(w)^{(i-1)})$ and $N_k(f(w)^{(i)})=N_{k+1}(f(w)^{(i)})$ we can conclude that $|j-k|=1$ if and only if $|r_i-r_{i+1}|=1$.

{\bf Case 3:} We have  $w_iw_{i+1}={jk}$ for some $j$ and $k$. This means 
$$r_ir_{i+1}=(j+b-N_j(w^{(i-1)})-1)(k+b-N_k(w^{(i)})-1).$$ 
If $j=k$ then $N_j(w^{(i-1)})=N_k(w^{(i)})-1$ so $r_ir_{i+1}=J(J-1)$ for some $J$. If $j<k$ then because 
$N_j(w^{(i-1)})\geq N_k(w^{(i-1)})=N_k(w^{(i)})$ we have that $r_i<r_{i+1}$. In this case where $j<k$ if 
$N_j(w^{(i-1)})=N_k(w^{(i-1)})$, then $|r_i-r_{i+1}|=1$. If instead $N_j(w^{(i-1)})>N_k(w^{(i-1)})$, then $|r_i-r_{i+1}|>1$.
 If $j>k$ then $N_j(w^{(i-1)})+1=N_j(w^{(i)})\leq N_k(w^{(i)})$ which implies that $r_i>r_{i+1}$ and $|r_i-r_{i+1}|>1$.

{\bf Case 4:} We have  $w_iw_{i+1}$ equals $\ul{j}\ol{k}$ or $\ol{k}\ul{j}$ for some $j$ and $k$.  In either case  we have 
$$\ul{j}\mapsto j+b-N_j(w^{(i-1)})=J\text{ and } \ol{k}\mapsto k+N_k(f(w)^{(i-1)})=K.$$ 
The values $N_1(w^{(i-1)})\geq N_2(w^{(i-1)})\geq \cdots$ form an integer partition $\lambda=\lambda_1\lambda_2\cdots$ with $\lambda_j=N_j(w^{(i-1)})$. Further the value $N_k(f(w)^{(i-1)})$ can be determined by a part of the complement $\lambda'$ of $\lambda$, specifically $N_k(f(w)^{(i-1)})=\lambda'_{b-k+1}$. Because of the structure of $w$ we know that $N_j(w^{(i-1)})=N_{j+1}(w^{(i-1)})$ and $N_k(f(w)^{(i-1)})=N_{k+1}(f(w)^{(i-1)})$, so $\lambda_j=\lambda_{j+1}$ and $\lambda'_{b-k-1}=\lambda'_{b-k}$. This means that $\lambda_j\neq b-k$. Consider the case where $\lambda_j<b-k$ then $\lambda'_{b-k}>j$. Thus, $J>j+k$ and $K<j+k$. If instead $\lambda_j>b-k$ then $\lambda'_{b-k}<j$, so $J<j+k$ and $K>j+k$. This means that $|J-K|>1$ in all cases.

{\bf Case 5:} We have  $w_iw_{i+1}$ equals $\ul{j}{k}$ or $k\ul{j}$ for some $j$ and $k$. Because of the conditions on $w$ we can not have the case where $j\ul{j}$ or $\ul{j}(j+1)$.  In either case  we have 
$$\ul{j}\mapsto j+b-N_j(w^{(i-1)})=J\text{ and } k\mapsto k+b-N_k(w^{(i-1)})-1=K.$$ 
If $j=k$ then $w_iw_{i+1}=\ul{j}j\mapsto J(J-1)$. If $j<k$ then $N_j(w^{(i-1)})\geq N_k(w^{(i-1)})$. Because $N_k(w^{(i-1)})=N_k(w^{(i+1)})-1$ we actually have $N_j(w^{(i-1)})> N_k(w^{(i-1)})$,  so $J< K$.  In this case where $j<k$ if $k=j+1$, then $|J-K|=1$, which means that $w_iw_{i+1}=(j+1)\ul{j}$. If instead $k>j+1$, then $|J-K|>1$.
 Now suppose that $j>k$ so $N_j(w^{(i-1)})\leq N_k(w^{(i-1)})$, which quickly implies that $J>K$ and $|J-K|>1$.

{\bf Case 6:} We have  $w_iw_{i+1}$ equals $\ol{j}{k'}$ or $k'\ol{j}$ for some $j$ and $k'$. This is equivalent to the case where $z_iz_{i+1}$ equals $\ol{j}{k}$ or $k\ol{j}$ for some $j$ and $k$. 
Because of the conditions on $z$ we can not have the case where $\ol{j}j$ or $(j+1)\ol{j}$. First consider the case where $j=k$ so $z_iz_{i+1}=j\ol{j}\mapsto J(J+1)$ for some $J$. In any other case $j\neq k$ and   we have 
$$\ol{j}\mapsto j+N_j(z^{(i-1)})=J\text{ and } k\mapsto k+N_k(z^{(i-1)})=K.$$ 
 If $j<k$ then $N_j(z^{(i-1)})\leq N_k(z^{(i-1)})$. This implies $J<K$. In this case where $j<k$ if $k=j+1$, then $|J-K|=1$, which implies that $z_iz_{i+1}=\ol{j}(j+1)\mapsto J(J+1)$. If instead $k>j+1$, then $|J-K|>1$. 
  Now consider when $j>k$, then $N_j(z^{(i-1)})\geq N_k(z^{(i-1)})$. Specifically because $N_k(z^{(i+1)})=N_k(z^{(i-1)})+1$ we have $N_j(z^{(i-1)})> N_k(z^{(i-1)})$. Hence, $J>K$ and $|J-K|>1$. 
 

{\bf Part 2:}  From the above arguments $w_iw_{i+1}$ or $z_iz_{i+1}$ maps to $r_ir_{i+1}=J(J+1)$ under the following circumstances: $w_iw_{i+1}$ equals $\ul{j(j+1)}$, $\ol{j(j+1)}$ or  ${j(j+1)}$ where $N_j(w^{(i-1)})=N_{j+1}(w^{(i-1)})$ or $z_iz_{i+1}$ equals $\ol{j}(j+1)$ or $j\ol{j}$. Also, from the above arguments $w_iw_{i+1}$ maps to $r_ir_{i+1}=J(J-1)$ under the following circumstances: $w_iw_{i+1}$ equals $\ul{j(j-1)}$, $\ol{j(j-1)}$,  ${jj}$,   $\ul{j}j$ or $j\ul{(j-1)}$. Considering all possible combinations to get $r_ir_{i+1}r_{i+2}=p(p\pm 1)p$, we will find only four do not give us a contradiction. We only do not get a contradiction for the cases described in 2(abcd) of the statement.  

{\bf Injective:} Finally, we will show that $\psi$ is injective. Suppose that $\psi(w)=\psi(w')=r$ and $w\neq w'$. This means that there exists an $i$ with $w^{(i-1)}=w'^{(i-1)}$ and $w_i\neq w'_i$. If $w_i$ and $w'_i$ are both letters in $\ul{[a-1]}$, then we have a contradiction. By our map $w_i\neq w'_i$ would imply two different outputs for the $i$th letter in $r$.  We would similarly have a contradiction if  $w_i\neq w'_i$ are both in $\ol{[b-1]}$ or ${[a]}$. If one of $w_i\neq w'_i$ is in $\ul{[a-1]}$ and the other is in $\ol{[b-1]}$ then we also have a contradiction. This is for the following reason. By the argument in Case 4 if there are adjacent $\ul{j}$ and $\ol{k}$ in $w$, then their outputs under $\psi$ differ by at least two. This also implies that the output of $\ul{j}$ in $w^{(i-1)}\ul{j}$ and the output of $\ol{k}$ in $w^{(i-1)}\ol{k}$ differ by at least two for any $j$ and $k$. 
Say $w_i=\ul{j}$ and $w'_i=k$, then $r_i=j+b-N_j(w^{(i-1)})=k+b-N_k(w^{(i-1)})-1$. If $j=k$, then the outputs of $w_i$ and $w'_i$ must be different. If $j>k$ then $N_j(w^{(i-1)})\leq N_k(w^{(i-1)})$ and we have a contradiction.  If $j<k$ then $N_j(w^{(i-1)})\geq N_k(w^{(i-1)})$ and we must have that $k=j+1$. Because of our conditions on $w$ if $w_i=j$ then $N_j(w^{(i-1)})\geq N_{j+1}(w^{(i-1)})$, which means it is impossible for $w'=j+1$. 
Finally consider the case where  $w_i=\ol{j}$ and $w'_i=k'$ or equivalently $z_i=\ol{j}$ and $z'_i=k$, then $r_i=j+N_j(z^{(i-1)})=k+N_k(z^{(i-1)})$. Suppose $j=k$. Because $z_i=\ol{j}$ we know that $N_j(z^{(i-1)})= N_{j+1}(z^{(i-1)})$, which makes it impossible for $z_i'=j$. If $j>k$ then $N_j(z^{(i-1)})\geq N_k(z^{(i-1)})$ implying that $z_i$ and $z'_i$ have two different outputs. We get a similar contradiction if $j<k$. Hence, in all cases we have proven we have a contradiction so $\psi$ must be injective. 
\end{proof}

\def\arraystretch{1.3}
\begin{table}

\begin{center}
\begin{tabular}{|c|c|c|}
\hline
Input $w_iw_{i+1}$&Output of $\psi$&\\
\hline
\hline
$\ol{jk}$&$JK$ & $j-k=J-K$\\
\hline
\hline
$\ul{jk}$&$JK$ & $j-k=J-K$\\
\hline
\hline
${jj}$  &$J(J-1)$ &    \\
\hline
${j(j+1)}$  &$J(J+1)$ &  $N_j(w^{(i-1)}) = N_{j+1}(w^{(i-1)})$  \\
\hline
${j(j+1)}$  &$JK$ & $J<K-1$ and $N_j(w^{(i-1)}) > N_{j+1}(w^{(i-1)})$  \\
\hline
${jk}$  &$JK$ &  $j>k$ and $J-1>K$  \\
\hline
\hline
$\ul{j}\ol{k}$ or $\ol{k}\ul{j}$  &$JK$ &  $|J-K|>1$  \\
\hline
\hline
$\ul{j}j$   &$J(J-1)$ &  \\
\hline
$(j+1)\ul{j}$   &$(J+1)J$ &  \\
\hline
$\ul{j}k$   or $k\ul{j}$   &$JK$ or $KJ$ & $j<k-1$ and $J<K-1$ \\
\hline
$\ul{j}k$   or $k\ul{j}$   &$JK$ or $KJ$ & $j>k$ and $J-1>K$ \\
\hline
\hline
$z_iz_{i+1}=j\ol{j}$   &$J(J+1)$ &  \\
\hline
$z_iz_{i+1}=\ol{j}(j+1)$   &$J(J+1)$ &  \\
\hline
$z_iz_{i+1}$ equals $\ol{j}k$ or $k\ol{j}$   &$JK$ or $KJ$ & $j<k-1$ and $J<K-1$  \\
\hline
$z_iz_{i+1}$ equals $\ol{j}k$ or $k\ol{j}$   &$JK$ or $KJ$ & $j>k$ and $J-1>K$  \\
\hline
\end{tabular}
\end{center}
\caption{This is the table of all cases of inputs and their associated outputs for $\psi$.}
\label{tab:21-cases}
\end{table}

We finally prove that the vertices $R(21[\alpha,\beta])$ are in bijection with $V_{\alpha,\beta}$ in the following theorem, which will help us easily construct paths in $G_{\pi}$.

\begin{thm}
\label{thm:21-graph-iso}
Let $\pi=21[\alpha,\beta]$, $|\alpha|=a$ and $|\beta|=b$. 
The map $\psi$ is a bijection between $V_{\alpha,\beta}$ and $R(\pi)$ and proves that $H_{\alpha,\beta}$ is isomorphic to $G_{\pi}$. 
\end{thm}

\begin{proof}
We have already shown that $\psi$ is injective in Lemma~\ref{lem:21-edge_iff}. If we form a graph on the image of $\psi$ by connecting vertices by commutation and long braid moves, Lemma~\ref{lem:21-edge_iff} proves that we get a graph that is isomorphic to $H_{\alpha,\beta}$. Call the graph on the image $I$. We will only have to show two things to prove $H_{\alpha,\beta}$ is isomorphic to $G_{\pi}$, so also $\psi$ is a bijection between $V_{\alpha,\beta}$ and $R(\pi)$. The first thing is that we will find some specific $\tilde w$ that maps to $\tilde r$ that is in $R(\pi)$. Because $G_{\pi}$ is a connected graph that can be generated by one single vertex $\tilde r$ by using commutation and long braid moves, we can conclude that one connected component of $I$ is isomorphic to $G_{\pi}$. The second thing we will show is that $H_{\alpha,\beta}$ is connected, which proves that $H_{\alpha,\beta}$ is isomorphic to $G_{\pi}$. 

First we will show there exists some $\tilde w\in V_{\alpha,\beta}$ where  $\psi(\tilde w)=\tilde r$  is in $R(\pi)$. Let $\tilde u\in R(\alpha)$, $\tilde v\in R(\beta)$ and 
$\tilde x=1^b2^b\ldots a^b\in \Ballot_{a,b}$. Let $\tilde w$ be the concatenation $\ul{\tilde u}\ol{\tilde v} \tilde x$, which is in $V_{\alpha,\beta}$. Let 
$$\tilde r = (\tilde u_1+b)(\tilde u_2+b)\ldots (\tilde u_{\ell(\alpha)}+b)
\tilde v_1\tilde v_2\ldots \tilde v_{\ell(\beta)}
b(b-1)\ldots 1 (b+1)b\ldots 2 \ldots (a+b-1)(a+b-2)\ldots a.$$
We can see that $\tilde r$ is in $R(\pi)$ by seeing the tranformation on the identity $\iota_{a+b}$. After applying the first part,  $(\tilde u_1+b)(\tilde u_2+b)\ldots (\tilde u_{\ell(\alpha)}+b)$, the identity becomes $12\ldots b (\alpha_1+b)(\alpha_2+b)\ldots (\alpha_a+b)$. After applying $\tilde v_1\tilde v_2\ldots \tilde v_{\ell(\beta)}$ we get $\beta_1\beta_2\ldots \beta_b (\alpha_1+b)(\alpha_2+b)\ldots (\alpha_a+b)$. The last portion of $\tilde r$ switches the $\alpha$ and $\beta$ subsequences and we get $(\alpha_1+b)(\alpha_2+b)\ldots (\alpha_a+b)\beta_1\beta_2\ldots \beta_b =\pi$. By applying the map $\psi$ we can show that $\psi(\tilde w)=\tilde r$. 


Next we will show that $G_{\alpha,\beta}$ is connected by proving that  from all $w\in V_{\alpha,\beta}$ we can describe a path to $\tilde w$. Let $w\in V_{\alpha,\beta}$, so $w$ is formed from a shuffle of some $u\in \ul{R}(\alpha)$, $v\in \ol{R}(\beta)$ and $x\in \Ballot_{a,b}$. We claim for now, and will prove later, that there is a path from $w$ to $uvy$ for some $y\in \Ballot_{a,b}$. Because $G_{\alpha}$ is connected we can perform moves to change $u$ to $\ul{\tilde u}$, and because $G_{\beta}$ is connected we can perform moves to change $v$ to $\ol{\tilde v}$. By the commutation rules given earlier for $V_{\alpha,\beta}$ we can transform $y$ to $\tilde x$ by performing commutation moves that move the  1's to the front, then the 2's, then the 3's and so on. If we can prove our claim then we have proven that $G_{\alpha,\beta}$ is connected and thus is isomorphic to $G_{\pi}$. 

Now we will prove our claim that there exists a path from $w$ to $uvy$ for some $y\in \Ballot_{a,b}$. We will do so with a recursive algorithm starting with $w=u^{[0]}v^{[0]}y^{[0]}$ where $u^{[0]}$ and $v^{[0]}$ are empty and $y^{[0]}=w$. Suppose for some $i\geq 0$ that $u^{[i]}$ is a prefix of $u$, $v^{[i]}$ is a prefix of $v$ and $y^{[i]}$ is a shuffle of the remaining letters of $u$ and $v$ in the same order and some ballot sequence so that $u^{[i]}v^{[i]}y^{[i]}\in V_{\alpha,\beta}$. In each step the number of letters in $y^{[i]}$ will decrease by one as we move one occurrence of $\ul{j}$ or $\ol{j}$ left and add it to $u^{[i]}$ or $v^{[i]}$. Once $y^{[i]}$ is a ballot sequence we have achieved the goal of our claim. We will now describe how you construct $u^{[i+1]}v^{[i+1]}y^{[i+1]}$. 

\noindent {\bf Case 1:} The first letter in $y^{[i]}$ that is not in $[a]$ is $\ul{j}$. Let ${\bf y}$ be the prefix of $y$ containing all letters before $\ul{j}$. We can preform commutation moves on ${\bf y}$ to get $1^{m_1}2^{m_2}\ldots a^{m_a}$ where $m_1\geq m_2\geq \ldots$. We also have that $m_j=m_{j+1}$. Next we use commutation moves to transform the subsequence $j^{m_j}(j+1)^{m_j}$ into $(j(j+1))^{m_j}$. So we have just transformed ${\bf y}$ into $1^{m_1}2^{m_2}\ldots (j-1)^{m_{j-1}}(j(j+1))^{m_j}(j+2)^{m_{j+2}}\ldots a^{m_a}$. We can use commutation moves to move $\ul{j}$ left past $(j+2)^{m_{j+2}}\ldots a^{m_a}$, long braid moves to move $\ul{j}$ further left past $(j(j+1))^{m_j}$ and some more commutation moves to move $\ul{j}$ left past $1^{m_1}2^{m_2}\ldots (j-1)^{m_{j-1}}$. Finally, we use more commutation moves to move $\ul{j}$ left past $v^{[i]}$. This means that $u^{[i+1]}=u^{[i]}\ul{j}$, $v^{[i+1]}=v^{[i]}$ and $y^{[i+1]}$ is $y^{[i]}$ with that first occurrence of $\ul{j}$ removed and the prefix adjusted as described. 

\noindent {\bf Case 2:} The first letter in $y^{[i]}$ that is not in $[a]$ is $\ol{j}$. Let ${\bf y}$ be the prefix of $y$ of all letters before $\ol{j}$, and let ${\bf z}$ be its image under $f$. We can preform commutation moves on ${\bf z}$ to get ${b}^{m_b}(b-1)^{m_{b-1}}\ldots 1^{m_1}$ where $m_b\geq m_{b-1}\geq \ldots$. We also have that $m_j=m_{j+1}$. We can use commutation moves to transform the subsequence $(j+1)^{m_j}j^{m_j}$ into $((j+1)j)^{m_j}$. So we have just transformed ${\bf z}$ into $b^{m_b}\ldots (j+2)^{m_{j+2}}((j+1)j)^{m_j}(j-1)^{m_{j-1}}\ldots 1^{m_1}$. We can use commutation moves to move $\ol{j}$ left past $(j-1)^{m_{j-1}}\ldots 1^{m_1}$,long  braid moves to move $\ol{j}$ further left past $((j+1)j)^{m_j}$ and some more commutation moves to move $\ol{j}$ left past $b^{m_b}\ldots (j+2)^{m_{j+2}}$. This means that $u^{[i+1]}=u^{[i]}$, $v^{[i+1]}=v^{[i]}\ol{j}$ and $y^{[i+1]}$ is $y^{[i]}$ with that first occurrence of $\ol{j}$ removed and the prefix adjusted as described, but expressed as a ballot sequence, rather than a reverse ballot sequence. 
\end{proof}
\subsection{Diameters of the Graphs for Inflations $21[\alpha,\iota_b]$ and the Longest Permutations}

In this section we prove recursive bounds for the diameters of the graphs for $21$-inflations of the form $21[\alpha,\iota_b]$ for some permutation $\alpha$. The recursive bounds will give us exact formulas for the diameters of graphs for $21[\alpha,1]$. Because the longest permutation $\delta_n$ can be written as such an inflation we will have a recursion for the diameter of the graphs for $\delta_n$ as well. 

Since the inflations we are considering are a subcollection of the inflations in Subsection~\ref{subsec:21-inflation_ecoding} we can use Theorem~\ref{thm:21-graph-iso} to describe a graph isomorphic to $G_{\pi}$. 
This graph is on the collection of vertices
$$V_{\alpha,\iota_b}=\bigcup_{x\in \Ballot_{a,b}, u\in \ul{R}(\alpha)} \ol{\Shuff}(x,u,\emptyset).$$
The associated graph graph $H_{\alpha,\iota_b}$ on $V_{\alpha,\iota_b}$ will have edges from the following relations between its vertices. 
\begin{enumerate}
\item Commutation moves or relations come from exchanging adjacent 
\begin{enumerate}
\item $w_iw_{i+1}=\ul{pq}$ if $|p-q|>1$
\item $w_iw_{i+1}=pq$ if either $p>q$ or $p<q$ and $N_p(w^{(i-1)})>N_q(w^{(i-1)})$
\item $w_iw_{i+1}=\ul{p}q$ or $w_iw_{i+1}=q\ul{p}$ if $|p-q|>1$ or $p>q$ 
\end{enumerate}
\item Long braid moves or relations come exchanging the following  occurrences on consecutive indices
\begin{enumerate}
\item in $w$ we exchange $\ul{p(p+ 1)p}$ and $\ul{(p+1)p(p+1)}$
\item  in $w$ we exchange $\ul{p}p(p+1)$ and  $p(p+1)\ul{p}$
\end{enumerate}
\end{enumerate} 
We will call moves of the type 1(a) and 2(a), which happen purely on the letters in $\ul{[a-1]}$, {\it $\alpha$-moves}. We will call moves of type 1(b) {\it ballot-moves}, which happen purely on the letters in $[a]$. All the remaining moves we will call {\it shift-moves}. 

\begin{cor}
\label{cor:21-graph-iso}
The graph $H_{\alpha,\iota_b}$ is isomorphic to $G_{\pi}$ for $\pi=21[\alpha,\iota_b]$.
\end{cor}
We are now ready to prove recursions on the diameters of graphs $G_{\pi}$,  $C_{\pi}$ and  $B_{\pi}$ where $\pi=21[\alpha,\iota_b]$. It will be helpful to define a few statistics on $V_{\alpha,\iota_b}$. Given $w\in V_{\alpha,\iota_b}$ let 
\begin{align*}
\Cshift(w)&=|\{(i,i'):i<i', w_i=j, w_{i'}=\ul{k}, j\neq k, j\neq k+1\}|,\\
\Bshift(w)&=|\{(i,i'):i<i', w_i=j, w_{i'}=\ul{j}\}|\text{ and}\\
\ballot(w)&=|\{(i,i'):i<i', w_i=j, w_{i'}=k, j>k\}|.\\
\end{align*}

\begin{exmp}
Let $\alpha=321$, $b=2$ and $w=\ul{1}1123\ul{2}2\ul{1}3$ in $V_{\alpha,\iota_2}$. Then $\Cshift(w)=3$, $\Bshift(w)=3$ and $\ballot(w)=1$. 
\end{exmp}

We first describe minimal paths between certain vertices in $G_{\pi}$ in Lemma~\ref{lem:21-path1} and Lemma~\ref{lem:21-path2}. 

\begin{lem}
\label{lem:21-path1}
Let $w\in V_{\alpha,\iota_b}$ be a shuffle of $u\in \ul{R}(\alpha)$ and $x\in \Ballot_{a,b}$. Let $\tilde x = (12\ldots a)^b\in \Ballot_{a,b}$. 
\begin{enumerate}[(a)]
\item There exists a path from $w$ to $u\tilde x$ with $\Cshift(w)+\binom{a}{2}\binom{b}{2}-\ballot(w)$ commutation steps plus $\Bshift(w)$ long  braid steps. 
\item There exists a path from $w$ to $\tilde x u$ with $\ell(\alpha)b(a-2)-\Cshift(w)+\binom{a}{2}\binom{b}{2}-\ballot(w)$ commutation steps plus $\ell(\alpha)b-\Bshift(w)$ long braid steps. 
\end{enumerate}
\end{lem}

\begin{proof}
Let $w\in V_{\alpha,\iota_b}$ be a shuffle of $u\in \ul{R}(\alpha)$ and $x\in \Ballot_{a,b}$. Let $\tilde x = (12\ldots a)^b\in \Ballot_{a,b}$. We describe the path from $w$ to $\tilde w = u\tilde x$ recursively. Let
$w=u^{[0]}y^{[0]}$ where $u^{[0]}$ is empty and $y^{[0]}=w$. Suppose for some $i\geq 0$ that $u^{[i]}$ is a prefix of $u$ and $y^{[i]}$ is a shuffle of the remaining letters of $u$  in the same order and some ballot sequence so that $u^{[i]}y^{[i]}\in V_{\alpha,\iota_b}$. In each step the number of letters in $y^{[i]}$ will decrease by one as we move one occurrence of $\ul{j}$  left and add it to $u^{[i]}$. We will end with $y^{[k]}$, which is a pure ballot sequence and is equal to $\tilde x$. We will now describe how you construct $u^{[i+1]}y^{[i+1]}$. 

{\bf Case 1:} The word $y^{[i]}$ contains some letter $\ul{j}$, and say that $\ul{j}$ is the first from left to right and ${\bf y}$ is the prefix of $y^{[i]}$ before $\ul{j}$. If ${\bf y}$ is empty then let $u^{[i+1]}=u^{[i]}\ul{j}$ and $y^{[i+1]}$ be $y^{[i]}$ with the $\ul{j}$ removed. If ${\bf y}$ is not empty then  ${\bf y}$ is starts with $1$ and has some maximal letter $m_1$. We can use ballot commutation moves to bring the first occurrences of $[m_1]$ to the front to form $(12\ldots m_1)$. Note that as we move $k\in[m_1]$ left into its place it is only shifting left past smaller letters and each shift increases $\ballot({\bf y})$ by one. 
We have now transformed ${\bf y}$ into $(12\ldots m_1){\bf y}'$. Repeat this process on ${\bf y'}$ until we transform ${\bf y}$ into ${\bf y}''=(12\ldots m_1)(12\ldots m_2)\cdots(12\ldots m_k)$ where $m_1\geq m_2\geq \ldots \geq m_k$. By our conditions on $w$ we know that $m_l\neq j$ for all $l$. Note that each of these commutation ballot moves have been increasing $\ballot({\bf y})$ by one. Next we will take the $\ul{j}$ and shift it left through ${\bf y}''$. When we shift $\ul{j}$ left past a $l\neq j,j+1$ we are decreasing $\Cshift(w)$ by one, and when we shift $\ul{j}$ past the adjacent pair $j(j+1)$ we are decreasing $\Bshift(w)$ by one. We have just described how to shift $\ul{j}$ left past ${\bf y}''$, so we append $\ul{j}$ to the right side of $u^{[i]}$ to form $u^{[i+1]}$, and we remove the $\ul{j}$  from $y^{[i]}$ and adjust the prefix to ${\bf y}''$ to get $y^{[i+1]}$. 

{\bf Case 2:} The word $y^{[i]}$ does not contain any letter $\ul{j}$, so is a ballot sequence. If $y^{[i]}$ is not $\tilde x$, use the process above, that we used to transform ${\bf y}$ into ${\bf y}''$, to transform $y^{[i]}$ into $\tilde x$. With this our recursive process is complete. 

Note that in the recursive process we have been increasing $\ballot(w)$ until $\ballot(w)$ equals $\ballot(\tilde x)=\binom{a}{2}\binom{b}{2}$. This means in this process we have used $\binom{a}{2}\binom{b}{2}-\ballot(w)$ commutation ballot-moves. Also in the recursive process we have been decreasing $\Cshift(w)$ until $\Cshift(w)$ equalled 0 by using $\Cshift(w)$  commutation shift-moves. Finally, in the recursive process we have been decreasing $\Bshift(w)$ until $\Bshift(w)$ equalled 0 by using $\Bshift(w)$  long braid shift-moves. Hence, we have described a path from $w$ to $u\tilde x$ with $\Cshift(w)+\binom{a}{2}\binom{b}{2}-\ballot(w)$ commutation steps plus $\Bshift(w)$ long braid steps. 

We can describe the path from $w$ to $w' = \tilde x u$ recursively using a very similar manner. Along this path the value of  $\ballot(w)$ will increase until $\ballot(w)$ equals $\ballot(\tilde x)=\binom{a}{2}\binom{b}{2}$, so will use $\binom{a}{2}\binom{b}{2}-\ballot(w)$ commutation ballot-moves. 
The path moves letters $\ul{j}$ of $u$ to the right, so along this path the value of   $\Cshift(w)$ will increase to its maximum value $\Cshift(\tilde xu)=\ell(\alpha)b(a-2)$ by using $\ell(\alpha)b(a-2)-\Cshift(w)$ commutation shift-moves. Again,  the  path moves letters $\ul{j}$ of $u$ to the right, so  along this path the value of   $\Bshift(w)$ will increase  to its maximum value $\Bshift(\tilde xu)=\ell(\alpha)b$ by using $\ell(\alpha)b-\Bshift(w)$ long braid shift-moves. Hence, exists a path from $w$ to $\tilde x u$ with $\ell(\alpha)b(a-2)-\Cshift(w)+\binom{a}{2}\binom{b}{2}-\ballot(w)$ commutation steps plus $\ell(\alpha)b-\Bshift(w)$ long braid steps. 
\end{proof}

\begin{lem}
\label{lem:21-path2}
Let $H_1$ be the induced subgraph of $H_{\alpha,\iota_b}$ on vertices $\{ux:u\in\ul{R}(\alpha), x\in\Ballot_{a,b}\}$, and $H_2$ be the  induced subgraph on vertices $\{xu:u\in\ul{R}(\alpha), x\in\Ballot_{a,b}\}$. 
\begin{enumerate}[(a)]
\item Any path from $w\in H_1$ to $w'\in H_2$ requires at least $\ell(\alpha)b(a-2)$ commutation steps plus $\ell(\alpha)b$ long braid steps. 
\item For any $w,w'\in V_{\alpha,\iota_b}$ where the ballot sequence in $w'$ is  $\tilde x = (12\ldots a)^b$, any path from $w$ to $w'$ has at least $\binom{a}{2}\binom{b}{2}-\ballot(w)$ ballot-steps. 
\end{enumerate}
\end{lem}
\begin{proof}
Let $H_1$ be the subgraph of $H_{\alpha,\iota_b}$ on vertices $\{ux:u\in\ul{R}(\alpha), x\in\Ballot_{a,b}\}$, and $H_2$ be the subgraph on vertices $\{xu:u\in\ul{R}(\alpha), x\in\Ballot_{a,b}\}$. Consider a path $P$ from $w\in H_1$ to $w'\in H_2$. Note that $\Cshift(w')=\ell(\alpha)b(a-2)$, $\Bshift(w')=\ell(\alpha)b$ and $\Cshift(w)=\Bshift(w)=0$. Also, note that any step in $P$ will be either an $\alpha$-move, ballot-move or a shift-move. A ballot or $\alpha$-move won't change the value of $\Cshift$ or $\Bshift$ and a shift-move will only change $\Cshift$ or $\Bshift$, but not both, by exactly one. This means in order to get from $w$ to $w'$ we require $\Cshift(w')=\ell(\alpha)b(a-2)$ commutation shift-moves and $\Bshift(w')=\ell(\alpha)b$ long braid shift-moves. 

Let $w,w'\in V_{\alpha,\iota_b}$ where the ballot sequence in $w'$ is  $\tilde x = (12\ldots a)^b$. Consider a path $P$ from $w$ to $w'$. Note that $\ballot(w')=\binom{a}{2}\binom{b}{2}$. Also, note that any step in $P$ will be either an $\alpha$-move, ballot-move or a shift-move. A shift-move or an $\alpha$-move won't change the value of $\ballot(w)$ and a ballot-move will only change $\ballot(w)$  by exactly one. This means in order to get from $w$ to $w'$ we require $\binom{a}{2}\binom{b}{2}-\ballot(w)$ commutation ballot-moves. 
\end{proof}

\begin{thm}
\label{thm:21-bounds}
Let $\pi=21[\alpha,\iota_b]$ and $|\alpha|=a$. 
\begin{enumerate}[(i)]
\item $\diam( G_{\alpha}) +  \ell(\alpha)b(a-1)+\binom{a}{2}\binom{b}{2}\leq \diam (G_{\pi})\leq \diam( G_{\alpha}) +  \ell(\alpha)b(a-1)+2\binom{a}{2}\binom{b}{2}$
\item $\diam (C_{\pi})=\diam (C_{\alpha}) + \ell(\alpha)b$
\item $\diam( B_{\alpha} )+ \ell(\alpha)b(a-2)+\binom{a}{2}\binom{b}{2}\leq \diam( B_{\pi})\leq \diam( B_{\alpha} )+ \ell(\alpha)b(a-2)+2\binom{a}{2}\binom{b}{2}$. 
\end{enumerate}
\end{thm}
\begin{proof}
By Corollary~\ref{cor:21-graph-iso} we know that $G_{\pi}$ is isomorphic to $H_{\alpha,\iota_b}$, so it suffices to prove this theorem on the graph $H_{\alpha,\iota_b}$. We will need to consider  specific subgraphs of $H_{\alpha,\iota_b}$. The first is on the vertices $\{ux:u\in\ul{R}(\alpha), x\in\Ballot_{a,b}\}$, which we will call $H_1$. The second is on the vertices $\{xu:u\in\ul{R}(\alpha), x\in\Ballot_{a,b}\}$, which we will call $H_2$.
In our proof we will be using two very particular ballot sequences $\tilde x = (12\ldots a)^b$ and $\tilde y=1^b2^b\ldots a^b$. 

Let $w,w'\in V_{\alpha,\iota_b}$ be shuffles of $u,u'\in \ul{R}(\alpha)$ and $x,x'\in \Ballot_{a,b}$ respectively. 
We can construct two different paths from $w$ to $w'$. By Lemma~\ref{lem:21-path1} we can construct one path $P_1$ through $H_1$ by starting at $w$, proceeding to $u\tilde x$, then $u'\tilde x$ and finally $w'$. 
This path uses 
$$\Cshift(w)+\Cshift(w')+2\binom{a}{2}\binom{b}{2}-\ballot(w)-\ballot(w')$$
commutation steps, 
$$\Bshift(w)+\Bshift(w')$$
long braid steps, and $d(u,u')$ $\alpha$-moves. By Lemma~\ref{lem:21-path1} we can construct a second path $P_2$ through $H_2$ by starting at $w$, proceeding to $\tilde x u$, then $\tilde x u'$ and finally $w'$. 
This path uses 
$$2\ell(\alpha)b(a-2)-\Cshift(w)-\Cshift(w')+2\binom{a}{2}\binom{b}{2}-\ballot(w)-\ballot(w')$$
commutation steps, 
$$2\ell(\alpha)b-\Bshift(w)-\Bshift(w')$$
long braid steps, and $d(u,u')$ $\alpha$-moves. Either $P_1$ or $P_2$ is bounded above by the average, which has $\ell(\alpha)b(a-2)+\ballot(w)+\ballot(w')$ commutation steps and $\ell(\alpha)b$ long braid steps. Note that  $d(u,u')$ is bounded above by $\diam(G_{\alpha})$. This means
$$d(w,w')\leq \diam(G_{\alpha})+\ell(\alpha)b(a-1) +2\binom{a}{2}\binom{b}{2}.$$
This proves the upper bound for the diameter of $G_{\pi}$. Using the same argument we can show that we have found a path from $w$ to $w'$ that uses at most $\diam(B_{\alpha})+\ell(\alpha)b(a-2) +2\binom{a}{2}\binom{b}{2}$ commutation moves and at most $\diam(C_{\alpha})+\ell(\alpha)b$ long braid moves. This gives us an upper bound for both $\diam(C_{\pi})$ and $\diam(B_{\pi})$

Now let $u,u'$ be two reduced words in $R(\alpha)$ with $d(u,u')=\diam(G_{\alpha})$. Consider the vertices $w=\ul{u}\tilde x$ and $w'=\tilde y\ul{u}'$ of $H_{\alpha,\iota_b}$ and $P$ a path between them. Because $w$ is in $H_1$ and $w'$ is in $H_2$ by Lemma~\ref{lem:21-path2} we know that we have at least $\ell(\alpha)b(a-2)$ shift-steps that are commutation steps plus $\ell(\alpha)b$ shift-steps that are long braid steps. The path $P$ will also contain at least $\ballot(w')=\binom{a}{2}\binom{b}{2}$ ballot steps that are commutation steps. Finally, we must have at least $\diam(G_{\alpha})$ $\alpha$-steps else we could project our path onto $G_{\alpha}$ and get a shorter path from $u$ to $u'$. All together this means that any path from $w$ to $w'$ has length at least 
$$\diam( G_{\alpha}) +  \ell(\alpha)b(a-1)+\binom{a}{2}\binom{b}{2}.$$
We can similarly argue that any path $w$ to $w'$ has at least $\diam( C_{\alpha}) +  \ell(\alpha)b$ long  braid moves and at least $\diam( B_{\alpha}) +  \ell(\alpha)b(a-2)+\binom{a}{2}\binom{b}{2}$ commutation moves. This proves the lower bounds on $C_{\pi}$ and $B_{\pi}$ respectively. 
\end{proof}

When $b=1$ in Theorem~\ref{thm:21-bounds}, we get the following corollary. 

\begin{cor}
\label{cor:21_formula}
Let $\pi=21[\alpha,1]$ and $|\alpha|=a$. 
\begin{enumerate}[(i)]
\item $\diam (G_{\pi})= \diam( G_{\alpha}) +  \ell(\alpha)(a-1)$
\item $\diam (C_{\pi})=\diam (C_{\alpha}) + \ell(\alpha)$
\item $\diam( B_{\pi})= \diam( B_{\alpha} )+ \ell(\alpha)(a-2)$
\end{enumerate}
\end{cor}

Because $\delta_{n+1}=21[\delta_{n},1]$, we get the following corollary.
\begin{cor}
We have the following recursions for the diameters $D(n)=\diam(G_{\delta_n})$, $C(n)=\diam(C_{\delta_n})$ and $B(n)=\diam(B_{\delta_n})$ with $D(1)=C(1)=B(1)=0$. 
\begin{enumerate}[(i)]
\item $D(n+1)=D(n)+(n-1)\binom{n}{2}$
\item $C(n+1)=C(n)+\binom{n}{2}$
\item $B(n+1)=B(n)+(n-2)\binom{n}{2}$
\end{enumerate}
\label{cor:longest_word}
\end{cor}

\section{Diameters of the Graphs for 312 or 231 Avoiding Permutations}
\label{sec:pattern_avoidance}
In this section, we first describe how we can use symmetries on a square to justify cases where graphs $G_{\pi}$, $C_{\pi}$ and $B_{\pi}$ are isomorphic. Next, we describe 312 or 231 pattern avoiding permutations in terms of 12-inflations and 21-inflations. We then find exact recursive formulas for the diameters of the graphs $G_{\pi}, C_{\pi}$ and $B_{\pi}$ where the permutation $\pi$ is 312-avoiding or 231-avoiding. In order to do this, we use Theorem~\ref{thm:12_formula} and Corollary~\ref{cor:21_formula}, the recursive formulas for the diameters of $G_{\pi}, C_{\pi}$ and $B_{\pi}$ when $\pi=12[\alpha,\beta]$ and $\pi=21[\alpha,1]$ for any permutations $\alpha$ and $\beta$.

\subsection{Symmetries}
\label{sec:symm}
There are many pairs on permutations $\pi$ and $\sigma$ where $G_{\pi}$ and $G_{\sigma}$ are isomorphic. We will discuss three such cases  related to symmetries on the square, the dihedral group. These graph isomorphisms will allow us to extend some of our previous results, and more quickly justify our families of permutations that achieve the upper bound of  Conjecture~\ref{conj:RR}.

There are eight operations in the dihedral group  that preserve the square, and all of them can be described as a rotation $r_{\theta}$ counter-clockwise by $\theta$ degrees or a reflection $r_m$ over a line with slope $m$ going through the center of the square. If we perform the same operation on our box diagram of a permutation $\pi$, then the output is another box diagram of a permutation, which we will notate as $r_{\theta}(\pi)$ and $r_m(\pi)$. We will only be interested in three of these operations, $r_{180^{\circ}}$, $r_1$ and $r_{-1}$. 

\begin{exmp}
Let $\pi=3241$ where $R(\pi)=\{1231,1213,2123\}$. Then $r_{180^{\circ}}(\pi)=4132$ and $R(r_{180^{\circ}}(\pi))=\{3213,3231,2321\}$. Also, $r_{1}(\pi)=4213$ and $R(r_{1}(\pi))=\{1321,3121,3212\}$. Finally, $r_{-1}(\pi)=2431$ and $R(r_{-1}(\pi))=\{3123,1323,1232\}$. We illustrate this in Figure~\ref{fig:symm}. 
\end{exmp}

\begin{figure}
\begin{center}
\begin{tikzpicture}[scale=0.8]
\draw[gray, dashed] (1,1) grid (4,4);
\filldraw (1,3) circle (2pt);
\filldraw (2,2) circle (2pt);
\filldraw (3,4) circle (2pt);
\filldraw (4,1) circle (2pt);
\node[below] at (2.5,0) {$\pi=3241$};
\end{tikzpicture}
\qquad
\qquad
\begin{tikzpicture}[scale=0.8]
\draw[gray, dashed] (1,1) grid (4,4);
\filldraw (1,4) circle (2pt);
\filldraw (2,1) circle (2pt);
\filldraw (3,3) circle (2pt);
\filldraw (4,2) circle (2pt);
\node[below] at (2.5,0) {$r_{180^{\circ}}(\pi)=4132$};
\end{tikzpicture}
\qquad
\qquad
\begin{tikzpicture}[scale=0.8]
\draw[gray, dashed] (1,1) grid (4,4);
\filldraw (1,4) circle (2pt);
\filldraw (2,2) circle (2pt);
\filldraw (3,1) circle (2pt);
\filldraw (4,3) circle (2pt);
\node[below] at (2.5,0) {$r_{1}(\pi)=4231$};
\end{tikzpicture}
\qquad
\qquad
\begin{tikzpicture}[scale=0.8]
\draw[gray, dashed] (1,1) grid (4,4);
\filldraw (1,2) circle (2pt);
\filldraw (2,4) circle (2pt);
\filldraw (3,3) circle (2pt);
\filldraw (4,1) circle (2pt);
\node[below] at (2.5,0) {$r_{-1}(\pi)=2431$};
\end{tikzpicture}
\end{center}
\caption{For $\pi=3241$, we see that $r_{180^{\circ}}(\pi)=4132$, $r_{1}(\pi)=4231$, and $r_{-1}(\pi)=2431$.}
\label{fig:symm}
\end{figure}
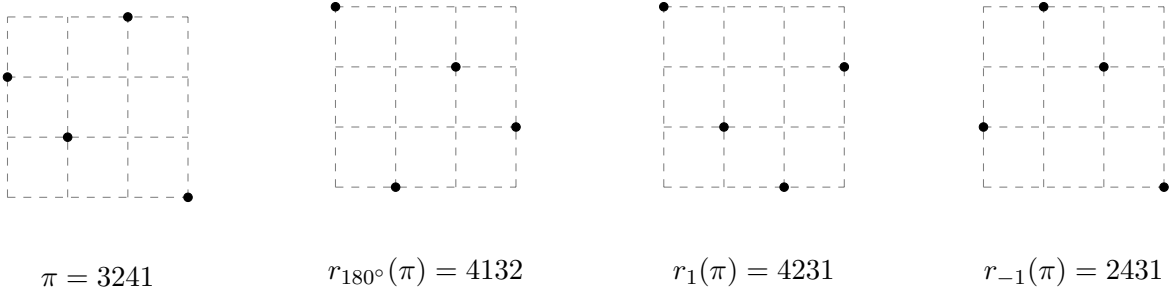

\begin{lem}
\label{lem:symmetries}
Let $\pi=\pi_1\pi_2\ldots\pi_n\in\fS_n$. 
\begin{enumerate}
\item Let $\sigma=r_{180^{\circ}}(\pi)$. Then, $\sigma=(n-\pi_n+1)(n-\pi_{n-1}+1)\ldots(n-\pi_1+1)$.  Also, there is a bijection  $R(\pi)\rightarrow R(\sigma)$ defined by $r_1r_2\cdots r_{\ell}\mapsto (n-r_1)(n-r_2)\cdots(n-r_{\ell})$.
\item Let $\sigma=r_{1}(\pi)$. Then, $\sigma=\pi^{-1}$.  Also, there is a bijection  $R(\pi)\rightarrow R(\sigma)$ defined by $r_1r_2\cdots r_{\ell}\mapsto r_{\ell}r_{\ell-1}\cdots r_{1}$.
\item Let $\sigma=r_{-1}(\pi)$. Then, $\sigma=(r_{180^{\circ}}(\pi))^{-1}$.  Also, there is a bijection  $R(\pi)\rightarrow R(\sigma)$ defined by $r_1r_2\cdots r_{\ell}\mapsto (n-r_{\ell})(n-r_{\ell-1})\cdots (n-r_{1})$.
\end{enumerate}
\end{lem}

\begin{proof}
Let $\pi\in\fS_n$. First let us consider $\sigma=r_{180^{\circ}}(\pi)$. The point $(i,\pi_i)$ in the box diagram of $\pi$ becomes $(n-i+1,n-\pi_i+1)$ in $\sigma$. This means $\sigma=(n-\pi_n+1)(n-\pi_{n-1}+1)\ldots(n-\pi_1+1)$.  We claim that $r_{180^{\circ}}(\alpha\circ\beta)=r_{180^{\circ}}(\alpha)\circ r_{180^{\circ}}(\beta)$ for any permutations $\alpha,\beta\in\fS_n$ composing from left to right. If the claim is true, because $r_{180^{\circ}}(s_i)=s_{n-i}$ we then know that $R(\pi)\rightarrow R(\sigma)$ defined by $r_1r_2\cdots r_{\ell}\mapsto (n-r_1)(n-r_2)\cdots(n-r_{\ell})$ is a bijection. To prove the claim we note in general  for $\pi\in\fS_n$ that $i\mapsto n-\pi_{n-i+1}+1$ in $r_{180^{\circ}}(\pi)$. For $r_{180^{\circ}}(\alpha\circ\beta)$, $i$ will map to $n-(\alpha\circ\beta)_{n-i+1}+1=n-\beta_{\alpha_{n-i+1}}+1$. For $r_{180^{\circ}}(\alpha)\circ r_{180^{\circ}}(\beta)$, the permutation $r_{180^{\circ}}(\alpha)$ will map $i$ to  $n-\alpha_{n-i+1}+1$ and $r_{180^{\circ}}(\beta)$ will further map this to $n-\beta_{\alpha_{n-i+1}}+1$, which proves the claim.

Now let $\sigma=r_{1}(\pi)$. The point $(i,\pi_i)$ in the box diagram of $\pi$ becomes $(\pi_i,i)$ is $\sigma$. This implies that $\sigma=\pi^{-1}$. Further we have a bijection $R(\pi)\rightarrow R(\sigma)$ because if $r_1r_2\cdots r_{\ell}\in R(\pi)$ then $\pi=s_{r_1}s_{r_2}\cdots s_{r_{\ell}}$.  Certainly $\pi^{-1}=s_{r_{\ell}}s_{r_{\ell-1}}\cdots s_{r_{1}}$, so $r_{\ell}r_{\ell-1}\cdots r_{1}\in R(\sigma)$. 

Finally let us consider $\sigma=r_{-1}(\pi)$. Note that $r_{\-1}$ is equal to the composition $r_{180^{\circ}}\circ r_1$. Since we have already proved this for the maps $r_{180^{\circ}}$ and $r_1$ we are done. 
\end{proof}

Because of the bijections on reduced words defined in Lemma~\ref{lem:symmetries} we can see that for every commutation move and long braid move made on $r\in R(\pi)$ there will be a corresponding move in its image in $R(r_x(\pi))$. This preserves the graph structure, the edge type and justifies why the corresponding graphs of $\pi$ and $r_x(\pi)$ on reduced words, commutation classes and long braid classes are isomorphic, for $x\in \{180^{\circ},1,-1\}$. 

\begin{prop}
\label{prop:dihedral}
For the following pairs of permutations, $\pi$ and $\sigma$, the graphs $G_{\pi}$, $C_{\pi}$ and $B_{\pi}$ are isomorphic to  $G_{\sigma}$, $C_{\sigma}$ and $B_{\sigma}$ respectively, so the diameters are also equal.
\begin{enumerate}
\item $\pi$ and $\sigma=r_{180^{\circ}}(\pi)$
\item $\pi$ and $\sigma=r_{1}(\pi)$
\item $\pi$ and $\sigma=r_{-1}(\pi)$
\end{enumerate}
\end{prop}


\subsection{Diameters of the Graphs for 312-Avoiding Permutations}

Let $\pi$ be a 312-avoiding permutation with $|\pi|=n$. Let $m\in[n]$ be such that $\pi_m=1$. Since $\pi$ is 312-avoiding, we should have $\pi_i<\pi_j$ for all $i<m<j$. Thus, we can write $\pi$ as 
\begin{equation*}
\pi=12[21[\pi',1],\pi'']
\end{equation*}
for two 312-avoiding permutations $\pi'$ and $\pi''$ with $|\pi'|=m-1$ and $|\pi''|=n-m$. Figure~\ref{fig:312-avoiding} would be helpful to see the structure of the permutation $\pi$. Hence, we have the following recursive formulas.

\begin{figure}
\label{fig:312-avoiding}
\begin{center}
\begin{tikzpicture}[scale=0.5]
\draw[gray, dashed] (0,0) grid (10,10);
\filldraw[fill=white, draw=black] (0,1) rectangle (4,5);
\node at (2,3) {$\pi'$};
\filldraw (5,0) circle (2pt);
\node at (6,.5) {$\pi_m=1$};
\filldraw[fill=white, draw=black] (6,6) rectangle (10,10);
\node at (8,8) {$\pi''$};
\end{tikzpicture}
\end{center}
\caption{The structure of the 312-avoiding permutation $\pi=12[21[\pi',1],\pi'']$.}
\end{figure}
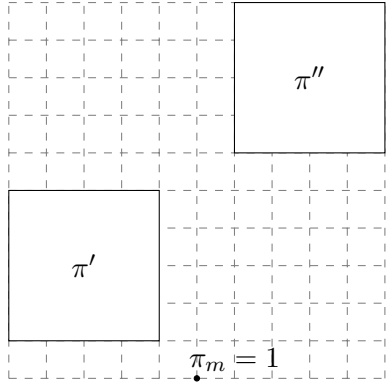

\begin{thm}
\label{312avoiding}
Let $\pi$ be a 312-avoiding permutation in $\fS_n$. Then, $\pi=12[21[\pi',1],\pi'']$ for $\pi'$ and $\pi''$ with $|\pi'|=m-1$ and $|\pi''|=n-m$. Moreover,
\begin{enumerate}[(i)]
\item $\diam(G_{\pi})=\diam(G_{\pi'})+\diam(G_{\pi''})+(m-1)(\ell(\pi')+\ell(\pi''))+\ell(\pi')(\ell(\pi'')-1)$
\item $\diam(C_{\pi})=\diam(C_{\pi'})+\diam(C_{\pi''})+\ell(\pi')$
\item $\diam(B_{\pi})=\diam(B_{\pi'})+\diam(B_{\pi''})+(m-1)(\ell(\pi')+\ell(\pi''))+\ell(\pi')(\ell(\pi'')-2)$
\end{enumerate}
\end{thm}

\begin{proof}
By Theorem~\ref{thm:12_formula} and Corollary~\ref{cor:21_formula}, we see that the diameter of $G_{\pi}$ is
\begin{align*}
\diam(G_{\pi})&=\diam(G_{21[\pi',1]})+\diam(G_{\pi''})+\ell(21[\pi',1])\ell(\pi'')\\
&=\diam(G_{\pi'})+\diam(G_{\pi''})+(|\pi'|-1)\ell(\pi')+\ell(\pi'')(|\pi'|+\ell(\pi'))\\
&=\diam(G_{\pi'})+\diam(G_{\pi''})+(m-1)(\ell(\pi')+\ell(\pi''))+\ell(\pi')(\ell(\pi'')-1).
\end{align*}
We also have the diameter of $C_{\pi}$ as follows.
\begin{align*}
\diam(C_{\pi})&=\diam(C_{21[\pi',1]})+\diam(C_{\pi''})\\
&=\diam(C_{\pi'})+\diam(C_{\pi''})+\ell(\pi').
\end{align*}
Lastly, we see that the diameter of $B_{\pi}$ is
\begin{align*}
\diam(B_{\pi})&=\diam(B_{21[\pi',1]})+\diam(B_{\pi''})+\ell(21[\pi',1])\ell(\pi'')\\
&=\diam(B_{\pi'})+\diam(B_{\pi''})+(|\pi'|-2)\ell(\pi')+\ell(\pi'')(|\pi'|+\ell(\pi'))\\
&=\diam(B_{\pi'})+\diam(B_{\pi''})+(m-1)(\ell(\pi')+\ell(\pi''))+\ell(\pi')(\ell(\pi'')-2).
\end{align*}
\end{proof}

\subsection{Diameters of the Graphs for 231-Avoiding Permutations}

Let $\pi$ be a 231-avoiding permutation with $|\pi|=n$. 
Let $m\in[n]$ be such that $\pi_m=n$. Since $\pi$ is 231-avoiding, we should have $\pi_i<\pi_j$ for all $i<m<j$. Thus, we can write $\pi$ as 
\begin{equation*}
\pi=12[\pi',21[1,\pi'']]
\end{equation*}
for two 231-avoiding permutations $\pi'$ and $\pi''$ with $|\pi'|=m-1$ and $|\pi''|=n-m$. Figure~\ref{fig:231-avoiding} would be helpful to see the structure of the permutation $\pi$. Hence, we have the following recursive formulas.

\begin{figure}
\label{fig:231-avoiding}
\begin{center}
\begin{tikzpicture}[scale=0.5]

\draw[gray, dashed] (0,0) grid (10,10);
\filldraw[fill=white, draw=black] (0,0) rectangle (4,4);
\node at (2,2) {$\pi'$};
\filldraw (5,10) circle (2pt);
\node at (6,9.5) {$\pi_m=n$};
\filldraw[fill=white, draw=black] (6,5) rectangle (10,9);
\node at (8,7) {$\pi''$};
\end{tikzpicture}
\end{center}
\caption{The structure of the 231-avoiding permutation $\pi=12[\pi',21[1,\pi'']]$.}
\end{figure}
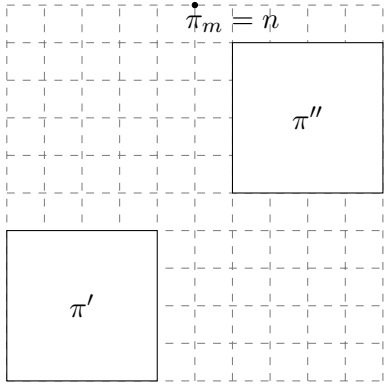

\begin{thm}
\label{thm:231avoiding}
Let $\pi$ be a 231-avoiding permutation in $\fS_n$. Then, $\pi=12[\pi',21[1,\pi'']]$ for $\pi'$ and $\pi''$ with $|\pi'|=m-1$ and $|\pi''|=n-m$. Moreover,
\begin{enumerate}[(i)]
\item $\diam(G_{\pi})=\diam(G_{\pi'})+\diam(G_{\pi''})+(n-m)(\ell(\pi')+\ell(\pi''))+\ell(\pi'')(\ell(\pi')-1)$
\item $\diam(C_{\pi})=\diam(C_{\pi'})+\diam(C_{\pi''})+\ell(\pi'')$
\item $\diam(B_{\pi})=\diam(B_{\pi'})+\diam(B_{\pi''})+(n-m)(\ell(\pi')+\ell(\pi''))+\ell(\pi'')(\ell(\pi')-2)$
\end{enumerate}
\end{thm}


\begin{proof}
Suppose that $\pi\in\fS_n$ avoids 231. This means that $\pi=12[\pi',21[1,\pi'']]$. Then, $\sigma=r_{180^{\circ}}(\pi)=12[21[\sigma'',1],\sigma']$ where $r_{180^{\circ}}(\pi')=\sigma'$ and $r_{180^{\circ}}(\pi'')=\sigma''$. By Proposition~\ref{prop:dihedral} we know that $G_{\pi}$, $C_{\pi}$ and $B_{\pi}$ are isomorphic to $G_{\sigma}$, $C_{\sigma}$ and $B_{\sigma}$, 
$G_{\pi'}$, $C_{\pi'}$ and $B_{\pi'}$ are isomorphic to $G_{\sigma'}$, $C_{\sigma'}$ and $B_{\sigma'}$ and 
$G_{\pi''}$, $C_{\pi''}$ and $B_{\pi''}$ are isomorphic to $G_{\sigma''}$, $C_{\sigma''}$ and $B_{\sigma''}$. Together with Theorem~\ref{312avoiding}, we are done. 
\end{proof}


\section{Connection to Hyperplane Arrangements}
\label{sec:RRconjecture}
In this section we connect our results to Conjecture~\ref{conj:RR}, Rienner and Roichman's conjecture. We prove that all permutations that avoid 231 or 312 satisfy the conjecture, and that these permutations achieve the upper bound of the conjecture. In addition, we found a set of permutations that achieve the lower bound of the conjecture. 

As in Section~\ref{sec:pre}, for an element $\pi\in \fS_n$, we define $I_2(\pi)$ to be the set of disjoint pairs of inversions $((i,j),(k,\ell))$ of $\pi$ and define $I_3(\pi)$ to be the set of all triples of inversions $((i,j),(i,k),(j,k))$ of $\pi$. Then $L_2(\pi)$ can be interpreted as the union of $I_2(\pi)$ and $I_3(\pi)$. Note that $|I_3(\pi)|$ is the number of 321 patterns in $\pi$. 

\begin{exmp}
Consider $\pi=4312\in \fS_4$. There are five inversions, those are $(1,2), (1,3), (1,4), (2,3)$, and $(2,4)$. Observe that there are two pairs of disjoint inversions, which are $((1,3), (2,4))$ and $((1,4) , (2,3))$, and $|I_2(\pi)|=2$. Note also that there are two 321 patterns, which are 431 and 432, and $|I_3(\pi)|=2$, thus $|L_2(\pi)|=|I_2(\pi)|+|I_3(\pi)|=4$.
\label{ex:L_2}
\end{exmp}

\begin{conj}[Rienner and Roichman]
\label{conj:RR}
For $\pi\in\fS_n$,
\begin{equation*}
\frac{1}{2}|L_2(\pi)|\leq \diam(G_{\pi})\leq |L_2(\pi)|.
\end{equation*}
\end{conj}

\subsection{12-Inflations or 21-Inflations}

We first confirm the conjecture for 12-inflations of two permutations $\alpha$ and $\beta$ and for 21-inflations of a permutation $\alpha$ and $\iota_b=12\dots b$. 

Suppose $\pi=12[\alpha,\beta]=\alpha_1\alpha_2\dots\alpha_a(\beta_1+a)(\beta_2+a)\dots(\beta_b+a)\in \fS_{a+b}$ is an 12-inflation of two permutations $\alpha\in \fS_a$ and $\beta\in \fS_b$. The disjoint pairs of inversions in $\pi$ fall under one of the three cases: (i) disjoint pairs in $\alpha$; (ii) disjoint pairs in $\beta$; (iii) one inversion in $\alpha$ and the other inversion in $\beta$. Thus, we have 
\begin{equation}
\label{eq:I_2(12)}
|I_2(\pi)|=|I_2(12[\alpha,\beta])|=|I_2(\alpha)|+|I_2(\beta)|+\ell(\alpha)\ell(\beta).
\end{equation}
Since the 321 patterns in $\pi$ are either 321 patterns in $\alpha$ or 321 patterns in $\beta$, we have
\begin{equation}
\label{eq:I_3(12)}
|I_3(\pi)|=|I_3(12[\alpha,\beta])|=|I_3(\alpha)|+|I_3(\beta)|.
\end{equation}
By equation~\eqref{eq:I_2(12)} and equation~\eqref{eq:I_3(12)}, we see that
\begin{align*}
|L_2(\pi)|&=|I_2(\pi)|+|I_3(\pi)|\\
&=|I_2(\alpha)|+|I_2(\beta)|+\ell(\alpha)\ell(\beta)+|I_3(\alpha)|+|I_3(\beta)|\\
&=|L_2(\alpha)|+|L_2(\beta)|+\ell(\alpha)\ell(\beta)
\end{align*}

The observation above together with the diameter formula $\diam(G_{\pi})=\diam(G_{\alpha})+\diam(G_{\beta})+\ell(\alpha)\ell(\beta)$ in Theorem~\ref{thm:12_formula} prove the following proposition.

\begin{prop}
If Conjecture~\ref{conj:RR} is true for both $\alpha$ and $\beta$, then the conjecture is also true for the 12-inflation of $\alpha$ and $\beta$. Moreover, if both $\diam(G_{\alpha})$ and $\diam(G_{\beta})$ hit the upper bound, then $\diam(G_{12[\alpha,\beta]})$ also hits the upper bound. 
\end{prop}

\begin{proof}
Suppose $\frac{1}{2}|L_2(\alpha)|\leq \diam(G_{\alpha})\leq |L_2(\alpha)|$ and $\frac{1}{2}|L_2(\beta)|\leq \diam(G_{\beta})\leq |L_2(\beta)|$. Then we see that
\begin{align*}
\diam(G_{\pi})&=\diam(G_{\alpha})+\diam(G_{\beta})+\ell(\alpha)\ell(\beta)\\
&\leq |L_2(\alpha)|+|L_2(\beta)|+\ell(\alpha)\ell(\beta)=|L_2(\pi)|
\end{align*}
and
\begin{align*}
\diam(G_{\pi})&=\diam(G_{\alpha})+\diam(G_{\beta})+\ell(\alpha)\ell(\beta)\\
&\geq \frac{1}{2}|L_2(\alpha)|+\frac{1}{2}|L_2(\beta)|+\frac{1}{2}\ell(\alpha)\ell(\beta)=\frac{1}{2}|L_2(\pi)|.
\end{align*}
In the case of $\diam(G_{\alpha})=|L_2(\alpha)|$ and $\diam(G_{\beta})=|L_2(\beta)|$, we see that
\begin{equation*}
\diam(G_{\pi})=\diam(G_{\alpha})+\diam(G_{\beta})+\ell(\alpha)\ell(\beta)=|L_2(\alpha)|+|L_2(\beta)|+\ell(\alpha)\ell(\beta)=|L_2(\pi)|.
\end{equation*}
\end{proof}

Suppose $\pi=21[\alpha,\iota_b]=(\alpha_1+b)(\alpha_2+b)\dots(\alpha_a+b)12\dots b\in \fS_{a+b}$ is the 21-inflation of a permutation $\alpha\in \fS_a$ and $\iota_b=12\dots b$. The disjoint pairs of inversions in $\pi$ fall under one of the three cases: (i) disjoint pairs in $\alpha$; (ii) one inversion $(i,j)$ in $\alpha$ and the other inversion $(r,s)$ for $r\in[a]-\{i,j\}, a+1\leq s\leq a+b$; (iii) two disjoint inversions $((i,r), (j,s))$ for $i,j\in[a]$ and $a+1\leq r,s\leq a+b$. Thus, we have 
\begin{equation}
\label{eq:I_2(21)}
|I_2(\pi)|=|I_2(21[\alpha,\iota_b])|=|I_2(\alpha)|+\ell(\alpha)(a-2)b+2{a\choose 2}{b\choose2}.
\end{equation}
Since the 321 patterns in $\pi$ are either 321 patterns in $\alpha$ or a subword $(\alpha_i+b)(\alpha_j+b)k$ of $\pi$ where $(i,j)$ is an inversion in $\alpha$ and $a+1\leq k\leq a+b$, we have
\begin{equation}
\label{eq:I_3(21)}
|I_3(\pi)|=|I_3(21[\alpha,\iota_b])|=|I_3(\alpha)|+\ell(\alpha)b.
\end{equation}
By equation~\eqref{eq:I_2(21)} and equation~\eqref{eq:I_3(21)}, we see that
\begin{align*}
|L_2(\pi)|&=|I_2(\pi)|+|I_3(\pi)|\\
&=|I_2(\alpha)|+\ell(\alpha)(a-2)b+2{a\choose 2}{b\choose2}+|I_3(\alpha)|+\ell(\alpha)b\\
&=|L_2(\alpha)|+\ell(\alpha)(a-1)b+2{a\choose2}{b\choose2}.
\end{align*}

The observation above together with the upper and lower bounds for the diameters of $G_{\pi}$, $\diam( G_{\alpha}) +  \ell(\alpha)b(a-1)+\binom{a}{2}\binom{b}{2}\leq \diam (G_{\pi})\leq \diam( G_{\alpha}) +  \ell(\alpha)b(a-1)+2\binom{a}{2}\binom{b}{2}$ in Theorem~\ref{thm:21-bounds} prove the following proposition.

\begin{prop}
If Conjecture~\ref{conj:RR} is true for $\alpha$, then the conjecture is also true for the 21-inflation of $\alpha$ and $\iota_b$ for any $b\geq 1$. 
\end{prop}

\begin{proof}
Suppose $\frac{1}{2}|L_2(\alpha)|\leq \diam(G_{\alpha})\leq |L_2(\alpha)|$. Then we see that
\begin{align*}
\diam(G_{\pi})&\leq \diam( G_{\alpha}) +  \ell(\alpha)b(a-1)+2\binom{a}{2}\binom{b}{2}\\
&\leq |L_2({\alpha})| +  \ell(\alpha)b(a-1)+2\binom{a}{2}\binom{b}{2}=|L_2(\pi)|
\end{align*}
and
\begin{align*}
\diam(G_{\pi})&\geq \diam( G_{\alpha}) +  \ell(\alpha)b(a-1)+\binom{a}{2}\binom{b}{2}\\
&\geq \frac{1}{2}|L_2(\alpha)|+\ell(\alpha)b(a-1)+\binom{a}{2}\binom{b}{2}=\frac{1}{2}|L_2(\pi)|.
\end{align*}
\end{proof}

Let us consider a special case $\pi=21[\alpha,1]$  when $b=1$. Observe that Conjecture~\ref{conj:RR} is true for $\pi$ since it is a special case of the previous proposition. Moreover, we get the next proposition by the diameter formula $\diam(G_{\pi})=\diam(G_{\alpha})+\ell(\alpha)(a-1)$ in Corollary~\ref{cor:21_formula}.

\begin{prop}
If $\diam(G_{\alpha})$ hits the upper bound, then $\diam(G_{21[\alpha,1]})$ also hits the upper bound. 
\end{prop}

\begin{proof}
Suppose $\diam(G_{\alpha})=|L_2(\alpha)|$. Observe that
\begin{equation*}
\diam(G_{\pi})=\diam(G_{\alpha})+\ell(\alpha)(a-1)=|L_2(\alpha)|+\ell(\alpha)(a-1)=|L_2(\alpha)|
\end{equation*}
and the proof follows.
\end{proof}

\subsection{Pattern Avoiding Permutations}

We show that Conjecture~\ref{conj:RR} is true for any permutations $\pi$ that avoid 231 or 312 patterns. To do this, it is sufficient to show that $\diam(G_{\pi})=|L_2(\pi)|$.

\begin{thm}
\label{thm:312-upper}
If $\pi$ is a 312-avoiding permutation, then $\diam(G_{\pi})=|L_2(\pi)|$.
\end{thm}

\begin{proof}
As in Section~\ref{sec:pattern_avoidance}, we can write $\pi=12[21[\pi',1],\pi'']$ 
 for two 312-avoiding permutations $\pi'$ and $\pi''$ with $\pi_m=1,|\pi|=n,|\pi'|=m-1$, and $|\pi''|=n-m$.

The disjoint pairs of inversions in $\pi$ fall under one of the five cases: (i) disjoint pairs in $\pi'$; (ii) disjoint pairs in $\pi''$; (iii) one inversion in $\pi'$ and the other inversion in $\pi''$; (iv) one inversion $(r,s)$ in $\pi''$ and the other inversion $(i,m)$ for $i\in[|\pi'|]=[m-1]$; (v) one inversion $(i,j)$ in $\pi'$ and the other inversion $(k,m)$ where $1\leq k\leq m-1$ and $k\notin\{i,j\}$. Thus, we have
\begin{equation}
\label{eq:I_2(312)}
|I_2(\pi)|=|I_2(\pi')|+|I_2(\pi'')|+\ell(\pi')\ell(\pi'')+\ell(\pi'')|\pi'|+\ell(\pi')(|\pi'|-2).
\end{equation}
Since the 321 patterns in $\pi$ are either 321 patterns in $\pi'$ or in $\pi''$ or a subword $(\pi'_i+1)(\pi'_j+1)1$ where $(i,j)$ is an inversion in $\pi'$, we have
\begin{equation}
\label{eq:I_3(312)}
|I_3(\pi)|=|I_3(\pi')|+|I_3(\pi'')|+\ell(\pi').
\end{equation}
By equation~\eqref{eq:I_2(312)} and equation~\eqref{eq:I_3(312)} we see that 
\begin{equation*}
|L_2(\pi)|=|L_2(\pi')|+|L_2(\pi'')|+(|\pi'|-1)\ell(\pi')+\ell(\pi'')(|\pi'|+\ell(\pi')).
\end{equation*}
We will prove the theorem by induction on the size of the permutations. (Base case) If $|\pi|\leq 2$, then $\diam(G_{\pi})=|L_2(\pi)|=0$. (Induction) Suppose $\diam(G_{\pi})=|L_2(\pi)|$ for all permutations with the size less than $n$. Then,
\begin{align*}
\diam(G_{\pi})&=\diam(G_{\pi'})+\diam(G_{\pi''})+(|\pi'|-1)\ell(\pi')+\ell(\pi'')(|\pi'|+\ell(\pi'))\\
&=|L_2(\pi')|+|L_2(\pi'')|+(|\pi'|-1)\ell(\pi')+\ell(\pi'')(|\pi'|+\ell(\pi'))\\
&=|L_2(\pi)|
\end{align*}
and the proof follows.
\end{proof}

\begin{thm}
\label{thm:231-upper}
If $\pi$ is a 231-avoiding permutation, then $\diam(G_{\pi})=|L_2(\pi)|$.
\end{thm}

\begin{proof} Note that because $r_{180^{\circ}}(321)=321$ a 321 pattern in $\pi$ will map to a 321 pattern in $r_{180^{\circ}}(\pi)$, so $|I_3(\pi)|=|I_3(r_{180^{\circ}}(\pi))|$. Also, since $r_{180^{\circ}}(21)=21$ the any pair of disjoint inversions will map to another pair of disjoint inversions in $r_{180^{\circ}}(\pi)$, so $|I_2(\pi)|=|I_2(r_{180^{\circ}}(\pi))|$. This means that $|L_2(\pi)|=|L_2(r_{180^{\circ}}(\pi))|$. 

Consider a permutation $\pi$ that avoids 231. Because  $r_{180^{\circ}}(231)=312$ we know that  $\sigma=r_{180^{\circ}}(\pi)$ avoids 312.  By Proposition~\ref{prop:dihedral} the graphs $G_{\pi}$ and $G_{\sigma}$ are isomorphic, so since $|L_2(\pi)|=|L_2(\sigma)|$ Theorem~\ref{thm:312-upper} completes the proof. 
\end{proof}

\subsection{Lower Bound}

We have seen that the following permutations hit the upper bound of Conjecture~\ref{conj:RR}.
\begin{enumerate}[(i)]
\item $\pi=12[\alpha,\beta]$ if $\diam(G_{\alpha})=|L_2(\alpha)|$ and $\diam(G_{\beta})=|L_2(\beta)|$.
\item $\pi=21[\alpha,1]$ if $\diam(G_{\alpha})=|L_2(\alpha)|$.
\item A permutation $\pi$ that is either 231-avoiding or 312-avoiding.
\end{enumerate}
\begin{rmk}
\label{rmk:upper}
One can come up with a question: Are all the permutations that hit the upper bound of the conjecture constructed in these three ways? The answer is no because we have a counter example, $\pi=2413\in\fS_4$. The permutation $\pi$ satisfies none of three conditions, but $\diam(G_{\pi})=|L_2(\pi)|$.
\end{rmk}
It is natural for us to pay attention to the lower bound of the conjecture, and we look for permutations $\pi$ such that $\diam(G_{\pi})=\frac{1}{2}|L_2(\pi)|$. We make a list of all permutations $\pi\in\fS_n$ for $n=4,5,6$ such that $\diam(G_{\pi})=\frac{1}{2}|L_2(\pi)|$ in Table~\ref{tab:low_bound}. We also express them in terms of 12-inflations and 21-inflations of $\iota_k=12\dots k$ for some $k\geq 1$. This observation suggests the following conjecture.

\begin{conj}
Let $\pi\in\fS_n$ be a permutation. Then the permutation can be written as $\pi=12[\iota_c,21[\iota_a,\iota_b],\iota_d]$ with $a\geq 2$ and $b\geq 2$ if and only if $\diam(G_{\pi})=\frac{1}{2}|L_2(\pi)|$.
\end{conj}

We state and prove the ``only if" direction of the conjecture in the following theorem.

\begin{thm}
\label{thm:low_bound}
Let $\pi\in\fS_n$ be a permutation. If we can write $\pi=12[\iota_c,21[\iota_a,\iota_b],\iota_d]$ with $a,b\geq 2$ and $c,d\geq 0$, then $\diam(G_{\pi})=\frac{1}{2}|L_2(\pi)|$.
\end{thm}

\begin{rmk}
For the permutation $\pi$ in Theorem~\ref{thm:low_bound}, we can have $\pi=21[\iota_a,\iota_b]$ when $c=0$ and $d=0$. 
\end{rmk}

\begin{table}

\begin{center}
\begin{tabular}{|c|c|c|c|c|c|}
\hline
\multicolumn{2}{|c|}{$\fS_4$} & \multicolumn{2}{|c|}{$\fS_5$} & \multicolumn{2}{|c|}{$\fS_6$} \\
\hline
\hline
$3412$& $21[\iota_2,\iota_2]$&$14523$&$12[\iota_1,21[\iota_2,\iota_2]]$ & $125634$& $12[\iota_2,21[\iota_2,\iota_2]]$\\
\hline
& &$34125$&$12[21[\iota_2,\iota_2],\iota_1]$ & $145236$& $12[12[\iota_1,21[\iota_2,\iota_2]],\iota_1]$\\
\hline
& &$34512$&$21[\iota_3,\iota_2]$ & $145623$& $12[\iota_1,21[\iota_3,\iota_2]]$\\
\hline
& &$45123$&$21[\iota_2,\iota_3]$ & $156234$& $12[\iota_1,21[\iota_2,\iota_3]]$\\
\hline
& & & & $341256$& $12[21[\iota_2,\iota_2],\iota_2]$\\
\hline
& & & & $345126$& $12[21[\iota_3,\iota_2],\iota_1]$\\
\hline
& & & & $345612$& $21[\iota_4,\iota_2]$\\
\hline
& & & & $451236$& $12[21[\iota_2,\iota_3],\iota_1]$\\
\hline
& & & & $456123$& $21[\iota_3,\iota_3]]$\\
\hline
& & & & $561234$& $21[\iota_2,\iota_4]$\\
\hline
\end{tabular}
\end{center}
\caption{All permutations $\pi$ such that $\diam(G_{\pi})=\frac{1}{2}|L_2(\pi)|$ in $\fS_n$ for $n=4,5,6$ and expressions in terms of 12-inflations and 21-inflations of $\iota_k$ for some $k\geq 1$.}
\label{tab:low_bound}
\end{table}

To prove Theorem~\ref{thm:low_bound}, we first state and prove the following lemmas.

\begin{lem}
\label{lem:diam(12[alpha,iota])}
$\diam(G_{12[\alpha,\iota_b]})=\diam(G_{12[\iota_b,\alpha]})=\diam(G_{\alpha})$ for any permutation $\alpha$.
\end{lem}

\begin{proof}
By applying Theorem~\ref{thm:12_formula}, we can see that
\begin{align*}
\diam(G_{12[\alpha,\iota_b]})&=\diam(G_{\alpha})+\diam(G_{\iota_b})+\ell(\alpha)\ell(\iota_b)\\
&=\diam(G_{\alpha})+0+\ell(\alpha)\cdot0\\
&=\diam(G_{\alpha}).
\end{align*}
Similarly, we can show $\diam(G_{12[\iota_b,\alpha]})=\diam(G_{\alpha})$.
\end{proof}

By the lemma above, we only need to work on $21[\iota_a,\iota_b]\in\fS_n$ from now.

\begin{lem}
\label{lem:L_2(21)}
Let $\pi=21[\iota_a,\iota_b]\in\fS_n$. Then, $|L_2(\pi)|=2{a\choose2}{b\choose2}$.
\end{lem}

\begin{proof}
Let $\pi=21[\iota_a,\iota_b]=(1+b)(2+b)\dots(a+b)12\dots b$. Since there are not any 321 patterns in $\pi$, we see that $|I_3(\pi)|=0$. Note that all inversions in $\pi$ are of the form $(i, j+b)$ for $j\in[a], i\in[b]$. Thus, the number $|I_2(\pi)|$ of disjoint pairs of inversions in $\pi$ is $ab(a-1)(b-1)/2=2{a\choose2}{b\choose2}$. Therefore, $|L_2(\pi)|=|I_2(\pi)|+|I_3(\pi)|=2{a\choose2}{b\choose2}$.
\end{proof}

\begin{lem}
\label{lem:diam(21)}
Let $\pi=21[\iota_a,\iota_b]\in\fS_n$. Then, $\diam(G_{\pi})={a\choose2}{b\choose2}$.
\end{lem}

\begin{proof}
In Section~\ref{subsec:21-inflation_ecoding} we discussed that the graph $G_{\pi}=G_{21[\iota_a,\iota_b]}$ is isomorphic to the graph $H_{\iota_a,\iota_b}$. The vertex set $V(H_{\iota_a,\iota_b})$ is the set $\Ballot_{a,b}=\{1^b2^b\dots a^b,\dots,(12\dots a)^b\}$ of all ballot sequences and the edges are formed from exchanging adjacent letters $w_iw_{i+1}=pq$ of $w\in \Ballot_{a,b}$ if either (i) $p>q$ or (ii) $p<q$ and $N_p(w^{(i-1)})>N_q(w^{(i-1)})$. Observe that the graph $H_{\iota_a,\iota_b}$ is a graded poset of rank ${a\choose 2}{b\choose2}$ with a unique maximum $(123\dots a)^b$ and a unique minimum $1^b2^b3^b\dots a^b$ by the rank function $\rho(w):=\ballot(w)=|\{(i,i'):i<i', w_i=j, w_{i'}=k, j>k\}|$. See Figure~\ref{fig:Ballot_poset} for an example of the Hasse diagram of $G_{21[\iota_3,\iota_2]}\cong H_{\iota_3,\iota_2}$. 

We claim that the diameter of $G_{\pi}$ is the rank ${a\choose2}{b\choose2}$ of the poset $G_{\pi}$. Since the distance between the maximum and minimum is ${a\choose2}{b\choose2}$, we have $\diam(G_{\pi})\geq {a\choose 2}{b\choose2}$. Suppose $u\neq v$ are any ballot sequences in $\Ballot_{a,b}$. If $u$ and $v$ are in a same chain, then $d(u,v)\leq {a\choose2}{b\choose2}$. Assume $u$ and $v$ are not in a same chain. Then we can construct two different paths from $u$ to $v$. Let the first path $P$ starts at $u$ to the maximum element and ends at $v$. Let the second path $Q$ starts at $u$ to the minimum element and ends at $v$. Either $P$ or $Q$ is bounded above by the rank of the poset ${a\choose 2}{b\choose 2}$, and we see $\diam(G_{\pi})\leq {a\choose 2}{b\choose 2}$. This shows $\diam(G_{\pi})={a\choose2}{b\choose2}$.
\end{proof}

\begin{figure}
\label{fig:Ballot_poset}
\begin{center}
\begin{tikzpicture}[scale=0.5]
\node (max) at (0,4) {$123123$};
  \node (121323) at (0,2) {$121323$};
  \node (112323) at (3,0) {$112323$};
  \node (121233) at (-3,0) {$121233$};
  \node (min) at (0,-2) {$112233$};
  \draw (min) -- (112323) -- (121323) -- (max);
  \draw (min) -- (121233) -- (121323);
\end{tikzpicture}
\end{center}
\caption{The Hasse diagram of $G_{21[\iota_3,\iota_2]}\cong H_{\iota_3,\iota_2}$.}
\end{figure}
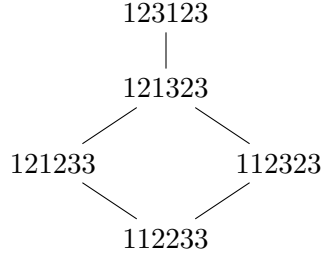

\begin{proof}[Proof of Theorem~\ref{thm:low_bound}]
Suppose $\pi=12[\iota_c,21[\iota_a,\iota_b],\iota_d]$ with $a,b\geq 2$ and $c,d\geq 0$. By Lemma~\ref{lem:diam(12[alpha,iota])}, we have $\diam(G_{\pi})=\diam(G_{21[\iota_a,\iota_b]})$. By Lemma~\ref{lem:L_2(21)} and Lemma~\ref{lem:diam(21)}, we have
\begin{equation*}
\diam(G_{\pi})={a\choose2}{b\choose2}=\frac{1}{2}|L_2(\pi)|
\end{equation*}
and this completes the proof.
\end{proof}

\section{Future Work}
\label{sec:future}
As we discussed in Remark~\ref{rmk:upper}, there are more families of permutations that achieve the upper bound of Conjecture~\ref{conj:RR}. This should be one direction we can study in the future.
\begin{op}
Find the necessary and sufficient  condition for permutations $\pi\in\fS_n$ to satisfy $\diam(G_{\pi})=|L_2(\pi)|$.
\end{op}
After investigating the diameters of $G_{\pi}$ for all permutations $\pi\in \fS_n$ for $n=4,5,6$, we make the following conjecture.
\begin{conj}
If $\pi\in\fS_n$ contains a pattern $3412$, then $\diam(G_{\pi})<|L_2(\pi)|$.
\end{conj}

\end{document}